\documentclass[a4paper]{article}

\usepackage[utf8]{inputenc}
\usepackage{lmodern}
\usepackage{geometry}
\usepackage{array}
\usepackage{graphicx}
\usepackage{amssymb, amsfonts, amsthm, amsmath}
\usepackage[english]{babel}
\usepackage[pdfborder={0 0 0}]{hyperref}
\usepackage{tikz}
\usetikzlibrary{decorations}
\usetikzlibrary{decorations.pathreplacing}
\tikzset{individu/.style={draw,thick}}

\usepackage{subfigure}

\theoremstyle{plain}
\newtheorem{theorem}{Theorem}[section]
\newtheorem{corollary}[theorem]{Corollary}
\newtheorem{lemma}[theorem]{Lemma}
\newtheorem{proposition}[theorem]{Proposition}

\theoremstyle{definition}

\theoremstyle{remark}
\newtheorem{remark}[theorem]{Remark}

\newcommand{\N}{\mathbb{N}}
\newcommand{\Z}{\mathbb{Z}}
\newcommand{\R}{\mathbb{R}}

\newcommand{\ind}[1]{\mathbf{1}_{\left\{#1\right\}}}

\newcommand{\floor}[1]{{\left\lfloor #1 \right\rfloor}}
\newcommand{\ceil}[1]{{\left\lceil #1 \right\rceil}}

\newcommand{\calC}{\mathcal{C}}

\numberwithin{equation}{section}

\DeclareMathOperator{\E}{\mathbf{E}}

\renewcommand{\P}{\mathbf{P}}
\DeclareMathOperator{\Var}{\mathbf{V}\mathrm{ar}}

\newcommand{\calT}{\mathcal{T}}

\newcommand{\calF}{\mathcal{F}}

\newcommand{\calL}{\mathcal{L}}

\newcommand{\wconv}[2]{{\underset{#1\to #2}{\Longrightarrow}}}

\renewcommand{\bar}[1]{\overline{#1}}

\newcommand{\T}{\mathbf{T}}
\renewcommand{\tilde}[1]{\widetilde{#1}}
\renewcommand{\hat}[1]{\widehat{#1}}

\renewcommand{\rho}{\varrho}
\renewcommand{\epsilon}{\varepsilon}

\title{Branching random walk with selection at critical rate}
\author{Bastien Mallein\footnote{LPMA, UPMC and DMA, ENS. Research partially supported by the ANR project MEMEMO2.}}
\date{\today}

\begin{document}

\maketitle

\begin{abstract}
We consider a branching-selection particle system on the real line. In this model the total size of the population at time $n$ is limited by $\exp\left(a n^{1/3}\right)$. At each step $n$, every individual dies while reproducing independently, making children around their current position according to i.i.d. point processes. Only the $\exp\left(a(n+1)^{1/3}\right)$ rightmost children survive to form the $(n+1)^\mathrm{th}$ generation. This process can be seen as a generalisation of the branching random walk with selection of the $N$ rightmost individuals, introduced by Brunet and Derrida in \cite{BrD97}. We obtain the asymptotic behaviour of position of the extremal particles alive at time $n$ by coupling this process with a branching random walk with a killing boundary.
\end{abstract}

\section{Introduction}
\label{sec:introduction}

Let $\calL$ be the law of a point process on $\R$. A branching random walk on $\R$ with reproduction law $\calL$ is a particle process defined as follows: it starts at time $0$ with a unique individual $\emptyset$ positioned at $0$. At time $1$, this individual dies giving birth to children which are positioned according to a point process of law $\calL$. Then at each time $k \in \N$, each individual in the process dies, giving birth to children which are positioned according to i.i.d. point processes of law $\calL$, shifted by the position of their parent. We denote by $\T$ the genealogical tree of the process, encoded with the Ulam-Harris notation. Note that $\T$ is a Galton-Watson tree. For a given individual $u \in \T$, we write $V(u) \in \R$ for the position of $u$, and $|u| \in \Z_+$ for the generation of $u$. If $u$ is not the initial individual, we denote by $\pi u$ the parent of $u$. The marked Galton-Watson tree $(\T,V)$ is the branching random walk on $\R$ with reproduction law $\calL$.

Let $L$ be a point process with law $\calL$. In this article, we assume the Galton-Watson tree $\T$ never get extinct and is supercritical, i.e.
\begin{equation}
  \label{eqn:reproduction}
  \P\left( \# L = 0 \right) = 0 \text{ } \text{ and } \text{ } \E\left[ \# L \right]>1.
\end{equation}
We also assume the branching random walk $(-V,\T)$ to be in the so-called boundary case, with the terminology of \cite{BiK04}:
\begin{equation}
  \label{eqn:critical}
  \E\left[ \sum_{\ell \in L} e^{\ell}\right]=1 ,\text{ } \E\left[ \sum_{\ell \in L} \ell e^{\ell} \right] = 0 \text{ } \text{ and } \text{ } \sigma^2 := \E\left[ \sum_{\ell \in L} \ell^2 e^{\ell} \right] < +\infty.
\end{equation}
Under mild assumptions, discussed in \cite[Appendix A]{Jaf12}, there exists an affine transformation mapping a branching random walk into a branching random walk in the boundary case. We impose the following integrability condition
\begin{equation}
  \label{eqn:integrability}
  \E\left[\sum_{\ell \in L} e^{\ell} \left[ \log \left( \sum_{\ell' \in L} e^{\ell'-\ell} \right) \right]^2 \right] < +\infty.
\end{equation}
Under slightly stronger integrability conditions, Aïdékon \cite{Aid13} proved that
\[
  \max_{|u|=n} V(u) + \frac{3}{2} \log n \wconv{n}{+\infty} W,
\]
where $W$ is a random shift of a Gumble distribution.

In \cite{BrD97}, Brunet and Derrida described a discrete-time particle system\footnote{Extended in \cite{BDMM07} to a particle system on $\R$.} on $\Z$ in which the total size of the population remains constant equal to $N$. At each time $k$, individuals alive reproduce in the same way as in a branching random walk, but only the $N$ rightmost individuals are kept alive to form the $(k+1)^\mathrm{th}$ generation. This process is called the $N$-branching random walk. They conjectured that the cloud of particles in the process moves at some deterministic speed $v_N$, satisfying
\[
  v_N = -\frac{\pi^2 \sigma^2}{2(\log N)^2} \left( 1 + \frac{(6 + o(1))\log \log N}{\log N} \right) \quad \text{as } N \to +\infty.
\]

Bérard and Gouéré \cite{BeG10} proved that in a $N$-branching random walk satisfying some stronger integrability conditions, the cloud of particles moves at linear speed $v_N$ on $\R$, i.e. writing $m^N_n, M^N_n$ respectively the minimal and maximal position at time $n$, we have
\[ \lim_{n \to +\infty} \frac{M^N_n}{n} = \lim_{n \to +\infty} \frac{m^N_n}{n} = v_N \text{ a.s. and } \lim_{N \to +\infty} (\log N)^2 v_N = -\frac{\pi^2 \sigma^2}{2}, \]
partially proving the Brunet-Derrida conjecture.

We introduce a similar model of branching-selection process. We set $\phi : \N \to \N$, and we consider a process with selection of the $\phi(n)$ rightmost individuals at generation $n$. More precisely we define $\T^\phi$ as a non-empty subtree of $\T$, such that $\emptyset \in \T^\phi$ and the generation $k \in \N$ is composed of the $\phi(k)$ children of $\{ u \in \T^\phi : |u|=k-1 \}$ with largest positions, with ties broken uniformly at random\footnote{Or in any other predictable fashion.}. The marked tree $(\T^\phi,V)$ is the branching random walk with selection of the $\phi(n)$ rightmost individuals at time $n$. We write
\begin{equation}
  \label{eqn:defineM}
  m^\phi_n = \min_{u \in \T^\phi, |u|=n} V(u) \quad \mathrm{and} \quad M^\phi_n = \max_{u \in \T^\phi, |u|=n} V(u).
\end{equation}
The main result of the article is the following.
\begin{theorem}
\label{thm:main}
Let $a>0$, we set $\phi(n) = \floor{\exp\left(a n^{1/3}\right)}$. Under assumptions \eqref{eqn:reproduction}, \eqref{eqn:critical} and \eqref{eqn:integrability} we have
\begin{equation}
  \label{eqn:convmin}
  M_n^\phi \sim_{n \to +\infty} -\frac{3\pi^2 \sigma^2}{2 a^2} n^{1/3} \quad \mathrm{a.s.}
\end{equation}
\begin{equation}
  \label{eqn:convmax}
  m_n^\phi \sim -\left(\frac{3\pi^2 \sigma^2}{2 a^2} n^{1/3} + a\right) n^{1/3} \quad \mathrm{a.s.}
\end{equation}
\end{theorem}

We prove Theorem \ref{thm:main} using a coupling between the branching random walk with selection and a branching random walk with a killing boundary, introduced in \cite{BeG10}. We also provide in this article the asymptotic behaviour of the extremal positions in a branching random walk with a killing boundary; and the asymptotic behaviour of the extremal positions in a branching random walk with selection of the $\floor{e^{h_{k/n}n^{1/3}}}$ at time $k \leq n$, where $h$ is a positive continuous function.

We consider in this article populations with $e^{a n^{1/3}}$ individuals on the interval of time $[0,n]$. This rate of growth is in some sense critical. More precisely in \cite{BDMM07}, the branching random walk with selection of the $N$ rightmost individuals is conjectured to typically behave at the time scale $(\log N)^3$. This observation has been confirmed by the results of \cite{BeG10,BBS13,Mai13}. Using methods similar to the ones developed here, or in \cite{BeG10}, one can prove that the maximal displacement in a branching random walk with selection of the $e^{a n^\alpha}$ rightmost individuals behaves as $-\frac{\pi^2 \sigma^2}{2(1-2\alpha)a^2} n^{1 - 2 \alpha}$ for $\alpha < 1/2$. If $\alpha>1/2$, it is expected that the behaviour of the maximal displacement in the branching random walk with selection is similar to the one of the classical branching random walk, of order $\log n$.

In this article, $c,C$ stand for positive constants, respectively small enough and large enough, which may change from line to line and depend only on the law of the processes we consider. Moreover, the set $\{|u| = n\}$ represents the set of individuals alive at the $n^\mathrm{th}$ generation in a generic branching random walk $(\T,V)$ with reproduction law $\calL$.

The rest of the article is organised as follows. In Section \ref{sec:lemmas}, we introduce the spinal decomposition of the branching random walk, the Mogul'ski\u\i{} small deviation estimate and lower bounds on the total size of the population in a Galton-Watson process. Using these results, we study in Section \ref{sec:killing_critical} the behaviour of a branching random walk with a killing boundary. Section \ref{sec:selection} is devoted to the study of branching random walks with selection, that we use to prove Theorem \ref{thm:main}.

\section{Some useful lemmas}
\label{sec:lemmas}
\subsection{The spinal decomposition of the branching random walk}

For any $a \in \R$, we write $\P_a$ for the probability distribution of $(\T,V+a)$ the branching random walk with initial individual positioned at $a$, and $\E_a$ for the corresponding expectation. To shorten notations, we set $\P=\P_0$ and $\E=\E_0$. We write $\calF_n = \sigma(u, V(u), |u| \leq n)$ for the natural filtration on the set of marked trees. Let $W_n = \sum_{|u|=n} e^{V(u)}$. By \eqref{eqn:critical}, we observe that $(W_n)$ is a non-negative martingale with respect to the filtration $(\calF_n)$. We define a new probability measure $\bar{\P}_a$ on $\calF_\infty$ such that for all $n \in \N$,
\begin{equation}
  \label{eqn:measurechange}
  \left.\frac{d\bar{\P}_a}{d\P_a}\right|_{\calF_n} = e^{-a} W_n.
\end{equation}
We write $\bar{\E}_a$ for the corresponding expectation and $\bar{\P}=\bar{\P}_0$, $\bar{\E} = \bar{\E}_0$. The so-called spinal decomposition, introduced in branching processes by Lyons, Pemantle and Peres in \cite{LPP95}, and extended to branching random walks by Lyons in \cite{Lyo97} gives an alternative construction of the measure $\bar{\P}_a$, by introducing a special individual with modified reproduction law.

Let $L$ be a point process with law $\calL$, we introduce the law $\hat{\calL}$ defined by
\begin{equation}
  \label{eqn:measurespine}
  \frac{d\hat{\calL}}{d\calL}(L) = \sum_{\ell \in L} e^{\ell}.
\end{equation}
We describe a probability measure $\hat{\P}_a$ on the set of marked trees with spine $(\T, V, w)$, where $(\T,V)$ is a marked tree, and $w = (w_n,n \in \N)$ is a sequence of individuals such that for any $n \in \N$, $w_n \in \T$, $|w_n|=n$ and $\pi w_n = w_{n-1}$. The ray $w$ is called the spine of the branching random walk.

Under law $\hat{\P}_a$, the process starts at time 0 with a unique individual $w_0=\emptyset$ located at position $a$. It generates its children according to a point process of law $\hat{\calL}$. Individual $w_1$ is chosen at random among the children $u$ of $w_0$ with probability proportional to $e^{V(u)}$. At each time $n \in \N$, every individual $u$ in the $n^\mathrm{th}$ generation die, giving independently birth to children according to the measure $\calL$ if $u \neq w_n$ and $\hat{\calL}$ if $u=w_n$. Finally, $w_{n+1}$ is chosen at random among the children $v$ of $w_n$ with probability proportional to $e^{V(v)}$.

\begin{proposition}[Spinal decomposition \cite{Lyo97}]
\label{pro:spinaldecomposition}
Under assumption \eqref{eqn:critical}, for all $n \in \N$, we have
\[ \left.\hat{\P}_a\right|_{\calF_n} = \left. \bar{\P}_a\right|_{\calF_n}. \]
Moreover, for any $u \in \T$ such that $|u|=n$,
\[ \hat{\P}_a\left( \left. w_n = u \right| \calF_n\right) = \frac{e^{V(u)}}{W_n}, \]
and $(V(w_n),n \geq 0)$ is a centred random walk starting from $a$ with variance $\sigma^ 2$
\end{proposition}

This proposition in particular implies the following result, often called in the literature the many-to-one lemma, which has been introduced for the first time by Kahane and Peyrière in \cite{KaP76,Pey74}, and links additive moments of the branching random walks with random walk estimates.
\begin{lemma}[Many-to-one lemma \cite{KaP76,Pey74}]
\label{lem:manytoone}
There exists a centred random walk $(S_n, n \geq 0)$, starting from $a$ under $\P_a$, with variance $\sigma^2$ such that for any $n \geq 1$ and any measurable non-negative function $g$, we have
\begin{equation}
  \label{eqn:manytoone}
  \E_a\left[ \sum_{|u|=n} g(V(u_1),\cdots V(u_n))\right] = \E_a\left[ e^{a-S_n} g(S_1,\cdots S_n) \right].
\end{equation}
\end{lemma}

\begin{proof}
We use Proposition \ref{pro:spinaldecomposition} to compute
\begin{align*}
  \E_a\left[ \sum_{|u|=n} g(V(u_1),\cdots V(u_n))\right]
  &= \bar{\E}_a \left[ \frac{e^{a}}{W_{n}} \sum_{|z|=n} g(V(u_1), \cdots V(u_n)) \right]\\
  &= \hat{\E}_a\left[ e^{a}\sum_{|u|=n} \ind{u =w_n} e^{-V(u)} g(V(u_1),\cdots V(u_n)) \right]\\
  &= \hat{\E}_a \left[ e^{a-V(w_n)} g(V(w_1), \cdots, V(w_n)) \right].
\end{align*}
Therefore we define the random walk $S$ under $\P_a$ as a process with the same law as $(V(w_n), n \geq 0)$ under $\hat{\P}_a$, which ends the proof.
\end{proof}

Using the many-to-one lemma, to compute the number of individuals in a branching random walk who stay in a well-chosen path, we only need to understand the probability for a random walk to stay in this path. This is what is done in the next section.

\subsection{Small deviation estimate and variations}

The following theorem gives asymptotic bounds for the probability for a random walk to have small deviations, i.e., to stay until time $n$ within distance significantly smaller than $\sqrt{n}$ from the origin. Let $(S_n, n \geq 0)$ be a centred random walk on $\R$ with finite variance $\sigma^2$. We assume that for any $x \in \R$, $\P_x(S_0=x) = 1$ and we set $\P=\P_0$.

\begin{theorem}[Mogul'ski\u\i{} estimate \cite{Mog74}]
\label{thm:mogulskii}
Let $f<g$ be continuous functions on $[0,1]$ such that $f_0<0<g_0$ and $(a_n)$ a sequence of positive numbers such that
\[ \lim_{n \to +\infty} a_n=+\infty \quad \mathrm{and} \quad \lim_{n \to +\infty} \frac{a_n^2}{n} = 0. \]
For any $f_1\leq x < y \leq g_1$, we have
\begin{equation}
  \label{eqn:mogulskii}
  \lim_{n \to +\infty} \frac{a_n^2}{n} \log \P\left[\frac{S_n}{a_n} \in [x,y], \frac{S_j}{a_n} \in \left[ f_{j/n}, g_{j/n}\right], j \leq n \right] = -\frac{\pi^2 \sigma^2}{2} \int_0^1 \frac{ds}{(g_s-f_s)^2}.
\end{equation}
\end{theorem}

In the rest of this article, we use some modifications of the Mogul'ski\u\i{} theorem, that we use later choosing $a_n=n^{1/3}$. We start with a straightforward corollary: the Mogul'ski\u\i{} theorem holds uniformly with respect to the starting point.
\begin{corollary}
\label{cor:mogulskii}
Let $f<g$ be continuous functions on $[0,1]$ such that $f_0<g_0$ and $(a_n)$ a sequence of positive numbers such that
\[
  \lim_{n \to +\infty} a_n=+\infty \quad \mathrm{and} \quad \lim_{n \to +\infty} \frac{a_n^2}{n} = 0.
\]
For any $f_1\leq x < y \leq g_1$, we have
\begin{equation}
  \label{eqn:cor_mogulskii}
  \lim_{n \to +\infty} \frac{a_n^2}{n} \log \sup_{z \in \R}\P_{za_n}\left[\frac{S_n}{a_n} \in [x,y], \frac{S_j}{a_n} \in \left[ f_{j/n}, g_{j/n}\right], j \leq n \right]
  = -\frac{\pi^2 \sigma^2}{2} \int_0^1 \frac{ds}{(g_s-f_s)^2}.
\end{equation}
\end{corollary}

\begin{proof}
We observe that
\begin{multline*}
  \sup_{z \in \R}\P_{z a_n}\left[\frac{S_n}{a_n} \in [x,y], \frac{S_j}{a_n} \in \left[ f_{j/n}, g_{j/n}\right], j \leq n \right]\\
  \geq \P_{a_n\frac{f_0+g_0}{2}}\left[ \frac{S_n}{a_n} \in [x,y], \frac{S_j}{a_n} \in \left[ f_{j/n}, g_{j/n}\right], j \leq n \right].
\end{multline*}
Therefore, applying Theorem \ref{thm:mogulskii}, we have
\begin{equation*}
  \liminf_{n \to +\infty} \frac{a_n^2}{n} \log \sup_{z \in \R}\P_{za_n}\left[\frac{S_n}{a_n} \in [x,y], \frac{S_j}{a_n} \in \left[ f_{j/n}, g_{j/n}\right], j \leq n \right]
  \geq -\frac{\pi^2 \sigma^2}{2} \int_0^1 \frac{ds}{(g_s-f_s)^2}.
\end{equation*}

We choose $\delta > 0$, and set $M = \ceil{\frac{g_0-f_0}{\delta}}$. We observe that
\[\P_{z a_n}\left[\frac{S_n}{a_n} \in [x,y], \frac{S_j}{a_n} \in \left[ f_{j/n}, g_{j/n}\right], j \leq n \right]=0,\]
thus
\begin{align*}
  &\sup_{z \in \R}\P_{z a_n}\left[\frac{S_n}{a_n} \in [x,y], \frac{S_j}{a_n} \in \left[ f_{j/n}, g_{j/n}\right], j \leq n \right]\\
  =& \max_{0 \leq k \leq M-1} \sup_{z \in [f_0+k\delta, f_0+(k+1)\delta]} \P_{z a_n}\left[\frac{S_n}{a_n} \in [x,y], \frac{S_j}{a_n} \in \left[ f_{j/n}, g_{j/n}\right], j \leq n \right]\\
  \leq& \max_{0 \leq k \leq M-1} \P_{a_n(f_0 + k \delta)}\left[ \frac{S_n}{a_n} \in [x, y + \delta], \frac{S_j}{a_n} \in \left[ f_{j/n}, g_{j/n}+\delta\right], j \leq n \right].
\end{align*}
As a consequence, we have
\begin{equation*}
  \limsup_{n \to +\infty} \frac{a_n^2}{n} \log \sup_{z \in \R}\P_{za_n}\left[\frac{S_n}{a_n} \in [x,y], \frac{S_j}{a_n} \in \left[ f_{j/n}, g_{j/n}\right], j \leq n \right]
  \leq -\frac{\pi^2 \sigma^2}{2} \int_0^1 \frac{ds}{(g_s-f_s+\delta)^2}.
\end{equation*}
Letting $\delta \to 0$ ends the proof.
\end{proof}

We present a more involved result on enriched random walks, a useful toy-model to study the spine of the branching random walk.  The following lemma is proved using a method similar to the original proof of Mogul'ski\u\i{}.
\begin{lemma}[Mogul'ski\u\i{} estimate for spine]
\label{lem:mogulskii_modified}
Let $((X_j, \xi_j), j \in \N)$ be an i.i.d. sequence of random variables taking values in $\R \times \R_+$, such that 
\[
  \E(X_1) = 0 \text{ } \text{ and } \text{ }  \sigma^2 := \E(X_1^2) < +\infty.
\]
We write $S_n = \sum_{j=1}^n X_j$ and $E_n = \{ \xi_j \leq n, j \leq n\}$. Let $(a_n) \in \R_+^\N$ be such that
\[
  \lim_{n \to +\infty} a_n = +\infty, \text{ } \lim_{n \to +\infty} \frac{a_n^2}{n} = 0 \text{ } \text{ and } \text{ } \lim_{n \to +\infty} a_n^2 \P(\xi_1 \geq n) = 0.
\]
Let $f<g$ be two continuous functions. For all $f_0<x<y<g_0$ and $f_1 < x'<y'<g_1$, we have
\[
  \lim_{n \to +\infty} \frac{a_n^2}{n} \inf_{z \in [x,y]} \log \P_{z a_n}\left(\frac{S_n}{a_n} \in [x',y'] \frac{S_j}{a_n} \in \left[ f_{j/n} ,g_{j/n} \right], j \leq n, E_n\right) = -\frac{\pi^2 \sigma^2}{2} \int_0^1 \frac{ds}{(g_s-f_s)^2}.
\]
\end{lemma}

\begin{proof}
For any $z \in [x,y]$, we have
\begin{equation*}
  \P_{z a_n}\left(\frac{S_n}{a_n} \in [x',y'], \frac{S_j}{a_n} \in \left[ f_{j/n} ,g_{j/n} \right], j \leq n, E_n\right)
  \leq \sup_{h \in \R} \P_{ha_n} \left(\frac{S_j}{a_n} \in \left[ f_{j/n} ,g_{j/n}\right], j \leq n\right).
\end{equation*}
So the upper bound in this lemma is a direct consequence of Corollary \ref{cor:mogulskii}. We now consider the lower bound.

We suppose in a first time that $f$ and $g$ are two constants. Let $n \geq 1$, $f<x<y<g$ and $f<x'<y'<g$, we bound from below the quantity
\[
  P^{x',y'}_{x,y} (f,g)
   = \inf_{z \in [x,y]} \P_{za_n} \left( \frac{S_n}{a_n} \in [x',y'], \frac{S_j}{a_n} \in [f,g], j \leq n, E_n \right).
\]
Setting $A\in \N$ and $r_n = \floor{A a_n^2}$, we divide $[0,n]$ into $K = \floor{\frac{n}{r_n}}$ intervals of length $r_n$. For $k \leq K$, we write $m_k= k r_n$, and $m_{K+1}=n$. By restriction to the set of trajectories verifying $S_{m_k} \in [x'a_n,y'a_n]$, and applying the Markov property at time $m_{K}, \ldots m_1$, and restricting to trajectories which are at any time $m_k$ in $[x'a_n,y'a_n]$, we have
\begin{equation}
  \label{eqn:beta}
  P_{x,y}^{x',y'}(f,g) \geq \pi^{x',y'}_{x,y}(f,g) \left(\pi^{x',y'}_{x',y'}\right)^{K},
\end{equation}
writing
\[
  \pi^{x',y'}_{x,y}(f,g) = \inf_{z \in [x,y]} \P_{za_n} \left( \frac{S_{r_n}}{a_n} \in [x',y'], \frac{S_j}{a_n} \in [f,g], j \leq r_n, E_{r_n} \right).
\]

Let $\delta > 0$ chosen small enough such that $M = \ceil{\frac{y-x}{\delta}} \geq 3$ we observe easily that
\begin{align}
  \pi^{x',y'}_{x,y}(f,g) &\geq \min_{0 \leq m \leq M} \pi_{x+m\delta,x+(m+1)\delta}^{x',y'}(f, g)\nonumber\\
  &\geq \min_{0 \leq m \leq M} \pi^{x'-(m-1)\delta, y-(m+1)\delta}_{x,x}(f-(m-1)\delta, g-(m+1)\delta).\label{eqn:alpha}
\end{align}
Moreover, we have
\begin{align*}
  \pi^{x',y'}_{x,x}(f,g) &= \P_{xa_n}\left( \frac{S_{r_n}}{a_n} \in [x',y'], \frac{S_j}{a_n} \in [f,g], E_{r_n} \right)\\
  &\geq \P_{xa_n}\left( \frac{S_{r_n}}{a_n} \in [x',y'], \frac{S_j}{a_n} \in [f,g] \right) - r_n\P(\xi_1 \geq n).
\end{align*}
Using the Donsker theorem \cite{Don51}, $\left(\frac{S_\floor{r_n t}}{a_n}, t \in [0,1]\right)$ converges, under law $\P_{xa_n}$, as $n \to +\infty$ to a Brownian motion with variance $\sigma \sqrt{A}$ starting from $x$. In particular
\[
  \liminf_{n \to +\infty} \pi^{x',y'}_{x,x}(f,g) \geq \P_x(B_{A\sigma^2} \in (x',y'), B_u \in (f,g), u \leq A \sigma^2).
\]

Using \eqref{eqn:alpha}, we have
\[
  \liminf_{n \to +\infty} \pi^{x',y'}_{x,y}(f,g) \geq \min_{0 \leq m \leq M}\P_{x+m\delta}(B_{A\sigma^2} \in (x'+\delta,y'-\delta), B_u \in (f+\delta,g-\delta), u \leq A \sigma^2).
\]
As a consequence, recalling that $K \sim \frac{n}{A a_n^2}$, \eqref{eqn:beta} leads to
\begin{multline}
  \liminf_{n \to +\infty} \frac{a_n^2}{n} \log P_{x,y}^{x',y'}(f,g)\geq \\
  \frac{1}{A}\min_{0 \leq m \leq M} \log \P_{x+m\delta}(B_{A\sigma^2} \in (x'+\delta,y'-\delta), B_u \in (f+\delta,g-\delta), u \leq A \sigma^2). \label{eqn:gamma}
\end{multline}
According to Karatzas and Shreve \cite{KaS91}, probability $\P_x(B_t \in (x',y'), B_s \in (f,g), s \leq t)$ is exactly computable, and
\[ \lim_{t \to +\infty} \frac{1}{t} \log_x \P(B_t \in (x',y'), B_s \in (f,g), s \leq t) = -\frac{\pi^2}{2(g-f)^2}. \]
Letting $A \to +\infty$ then $\delta \to 0$, \eqref{eqn:gamma} becomes
\begin{equation}
  \liminf_{n \to +\infty} \frac{a_n^2}{n} \log P_{x,y}^{x',y'}(f,g) \geq - \frac{\pi^2 \sigma^2}{2(g-f)^2}.
  \label{eqn:firstestimate}
\end{equation}

We now take care of the general case. Let $f<g$ be two continuous functions such that $f_0<0 < g_0$. We write $h_t = \frac{f_t+g_t}{2}$. Let $\epsilon>0$ be such that
\[
  12 \epsilon \leq \inf_{t \in [0,1]} g_t - f_t
\]
and $A \in \N$ such that
\[ \sup_{|t-s| \leq \frac{2}{A}} |f_t-f_s| + |g_t- g_s|  + |h_t-h_s| \leq \epsilon. \]
For any $a \leq A$, we write $m_a = \floor{an/A}$,
\[
  I_{a,A} = [f_{a/A}+\epsilon, g_{a/A}-\epsilon] \quad \mathrm{and} \quad J_{a,A} = [h_{a/A}-\epsilon, h_{a/A}+\epsilon],
\]
except $J_{0,A} = [x,y]$ and $J_{A,A} = [x',y']$.

We apply the Markov property at times $m_{A-1},\ldots, m_1$, we have
\begin{multline*}
  \inf_{z \in J_{0,A}} \P_{z a_n}\left(\frac{S_j}{a_n} \in \left[ f_{j/n} ,g_{j/n} \right], j \leq n, E_{n}\right)\\
  \geq \prod_{a=0}^{A-1} \inf_{z \in J_{a,A}} \P_{za_n}\left( \frac{S_{m_{a+1}}}{a_n} \in J_{a+1,A}, E_{m_{a+1}-m_a}
    \frac{S_j}{a_n} \in I_{a,A},j \leq m_{a+1}-m_a \right).
\end{multline*}
Applying equation \eqref{eqn:firstestimate}, we conclude
\begin{multline*}
  \liminf_{n \to +\infty} \frac{a_n^2}{n} \log \inf_{z \in J_{0,A}} \P_{z a_n} \left( \frac{S_j}{a_n} \in \left[ f_{j/n} ,g_{j/n} \right] \quad \mathrm{and} \quad \xi_j \leq n, j \leq n\right)\\
  \geq -\frac{1}{A}\sum_{a=0}^{A-1} \frac{\pi^2 \sigma^2}{2(g_{a,A}-f_{a,A}-2\epsilon)^2}.
\end{multline*}
Letting $\epsilon \to 0$ then $A \to +\infty$, we conclude the proof.
\end{proof}

Lemma \ref{lem:mogulskii_modified} is extended in the following fashion, to take into account functions $g$ such that $g(0)=0$.

\begin{corollary}
\label{cor:mogulskii_modified}
Let $((X_j, \xi_j), j \in \N)$ be an i.i.d. sequence of random variables taking values in $\R \times \R_+$ such that
\[
  \E(X_1) = 0  \text{ } \text{ and } \text{ } \sigma^2 := \E(X_1^2) < +\infty.
\]
We write $S_n = \sum_{j=1}^n X_j$ and $E_n = \left\{ \xi_j \leq n, j \leq n \right\}$. Let $(a_n) \in \R_+^\N$ verifying
\[
  \lim_{n \to +\infty} a_n = +\infty, \text{ } \limsup_{n \to +\infty} \frac{a_n^3}{n} < +\infty \text{ } \text{ and } \text{ } \lim_{n \to +\infty} a_n^2 \P(\xi_1 \geq n) = 0.
\]
Let $f<g$ be two continuous functions such that $f_0<0$ and $\liminf_{t \to 0} \frac{g_t}{t}>-\infty$. For any $f_1 \leq x' < y' \leq g_1$, we have
\[
  \lim_{n \to +\infty} \frac{a_n^2}{n} \log \P\left( \frac{S_n}{a_n} \in [x',y'], \frac{S_j}{a_n} \in [f_{j/n}, g_{j/n}],j \leq n, E_n \right) = - \frac{\pi^2 \sigma^2}{2}\int_0^1 \frac{ds}{(g_s-f_s)^2}.
\]
\end{corollary}

\begin{proof}
Let $d>0$ be such that for all $t \in [0,1]$, $g(t) \geq - d t$. We set $x < y < 0$ and $A>0$ verifying $\P(X_1 \in [x,y], \xi_1 \leq A)>0$. For any $\delta>0$, we set $N = \floor{\delta a_n}$. Applying the Markov property at time $N$, for any $n \in \N$ large enough, we have
\begin{multline*}
  \P\left( \frac{S_n}{a_n} \in [x',y'], \frac{S_j}{a_n} \in \left[ f_{j/n}, g_{j/n} \right], j \leq n, E_n \right)
  \geq \P\left( S_j \in [jx, jy], j \leq N, E_N \right) \\
  \times \inf_{z \in [2 \delta x, \delta y/2]}\P_{z a_n}\left(\frac{S_{n-N}}{a_n} \in [x', y '], \frac{S_{j-N}}{a_n} \in \left[ f_{\frac{j+N}{n}},g_{\frac{j+N}{n}} \right], j \leq n-N, E_{n-N} \right)
\end{multline*}
with $\P\left( S_j \in [jx, jy], j \leq N, E_N \right) \geq \P\left( X_1 \in [x,y], \xi_1 \leq A \right)^N$. As $\limsup_{n \to +\infty} \frac{a_n^3}{n} < +\infty$, we have
\begin{multline*}
  \liminf_{n \to +\infty} \frac{a_n^2}{n} \log \P\left( \frac{S_n}{a_n} \in [x',y'], \frac{S_j}{a_n} \in [f_{j/n}, g_{j/n}],j \leq n, E_n \right)\\
  \geq \liminf_{n \to +\infty} \frac{a_n^2}{n}\inf_{z \in [2 \delta x, \delta y/2]}\P_{z a_n}\left(\frac{S_{n-N}}{a_n} \in [x', y '], \frac{S_{j-N}}{a_n} \in \left[ f_{\frac{j+N}{n}},g_{\frac{j+N}{n}} \right], j \leq n-N, E_{n-N} \right).
\end{multline*}
Consequently, applying Lemma \ref{lem:mogulskii_modified} and letting $\delta \to 0$, we have
\[
  \liminf_{n \to +\infty} \frac{a_n^2}{n} \log \P\left( \frac{S_n}{a_n} \in [x',y'], \frac{S_j}{a_n} \in [f_{j/n}, g_{j/n}],j \leq n, E_n \right) \geq - \frac{\pi^2 \sigma^2}{2} \int_0^1 \frac{ds}{(g_s-f_s)^2}.
\]
The upper bound is a direct consequence of Corollary \ref{cor:mogulskii}.
\end{proof}

\subsection{Lower bounds for the total size of the population in a Galton-Watson process}

We start this section by recalling the definition of a Galton-Watson process. Let $\mu$ be a law on $\Z_+$, and $(X_{k,n},(k,n)\in \N^2)$ an i.i.d. array of random variables with law $\mu$. The process $(Z_n, n \geq 0)$ defined inductively by
\[
  Z_0 = 1 \quad \mathrm{and} \quad Z_{n+1} = \sum_{k=1}^{Z_n} X_{k,n+1}
\]
is a Galton-Watson process with reproduction law $\mu$. The quantity $Z_n$ represents the size of the population at time $n$, and $X_{k,n}$ the number of children of the $k^\text{th}$ individual alive at time $n$. Galton-Watson processes have been extensively studied since their introduction by Galton and Watson\footnote{Independently from the seminal work of Bienaymé, who also introduced and studied such a process in 1847.} in 1874. The results we use here can been found in \cite{AtN04}.

We write
\[
  f : \begin{array}{rcl}
  [0,1] & \to & [0,1]\\
  s &\mapsto & \E\left[ s^{X_{1,1}} \right] = \sum_{k=0}^{+\infty} \mu(k)s^k.
  \end{array}
\]
We observe that for all $n\in \N$, $\E\left(s^{Z_n}\right) = f^n(s)$, where $f^n$ is the $n^\text{th}$ iterate of $f$. Moreover, if $m:=\E(X_{1,1})<+\infty$, then $f$ is a $\mathcal{C}^1$ strictly increasing convex function on $[0,1]$ that verifies 
\[f(0) = \mu(0), \quad f(1)=1 \quad \mathrm{and} \quad f'(1) = m.\]
We write $q$ the smallest solution of the equation $f(q)=q$. It is a well-known fact that $q$ is the probability of extinction of the Galton-Watson process, i.e., $\P(\exists n \in \N : Z_n = 0) = q$. Observe that $q<1$ if and only if $m>1$. If $m>1$, we also introduce $\alpha :=- \frac{\log f'(q)}{\log m} \in (0,+\infty]$.
\begin{lemma}
\label{lem:gw}
Let $(Z_n, n \geq 0)$ be a Galton-Watson process with reproduction law $\mu$. We write $b = \min\{k \in \Z_+ : \mu(k) > 0\}$, $m=\E(Z_1) \in (1,+\infty)$ and $q$ for the smallest solution of the equation $\E(q^{Z_1}) = q$. There exists $C>0$ such that for all $z \in (0,1)$ and $n \in \N$, we have
\[
  \P(Z_n \leq z m^n) \leq
  \begin{cases}
    q + C z^{\frac{\alpha}{\alpha + 1}} &\text{if } b=0\\
    C z^{\alpha} & \text{if } b=1\\
    \exp\left[ - C z^{-\frac{\log b}{\log m - \log b}} \right] & \text{if } b \geq 2.
  \end{cases}
\]
\end{lemma}

\begin{remark}
One may notice that these estimates are in fact tight, under some suitable integrability conditions, uniformly in large $n$, as $z \to 0$. To obtain a lower bound, it is enough to compute the probability for a Galton-Watson tree to remain as small as possible until some time $k$ chosen carefully, then reproduce freely until time $n$. A more precise computation of the left tail of the Galton-Watson process can be found in \cite{FlW07}.
\end{remark}

\begin{proof}
We write $s_0= \frac{q+1}{2}$, and for all $k \in \Z$, $s_k = f^k(s_0)$, where negative iterations are iterations of $f^{-1}$. Using the properties of $f$, there exists $C_->0$ such that $1-s_k \sim_{k \to -\infty} C_- m^{k}$. Moreover, if $\mu(0) + \mu(1)>0$, there exists $C_+>0$ such that $s_k - q \sim_{k \to +\infty} C_+ f'(q)^k$. Otherwise,
\[
  s_k = f^{(b)}(0)^{\frac{b^k}{b-1} + o(b^k)} \quad \text{as } k \to +\infty
\]
where $f^{(b)}(0) = b! \mu(b)$ is the $b^\mathrm{th}$ derivative of $f$ at point 0.

Observe that for all $z < m^{-n}$, we have $\P(Z_n \leq z m^n) = \P(Z_n=0) \leq 1$. Therefore, we always assume in the rest of the proof that $z \geq m^{-n}$. By the Markov inequality, we have, for all $z \in (m^{-n},1)$ and $s \in (0,1)$,
\[
  \P(Z_n \leq z m^n) = \P(s^{Z_n} \geq s^{z m^n}) \leq \frac{\E(s^{Z_n})}{s^{z m^n}} =\frac{f^n_s}{s^{z m^n}}.
\]
In particular, for $s=s_{k-n}$, we have $\P(Z_n \leq z m^n) \leq \frac{s_k}{\left(s_{k-n}\right)^{z m^n}}$. The rest of the proof consists in choosing the optimal $k$ in this equation, depending on the value of $b$. 

If $b=0$, we choose $k = \frac{-\log z}{\log m - \log f'(q)}$ which grows to $+\infty$ as $z \to 0$, while $k \leq n\frac{\log m}{\log m - \log f'(q)}$ so $k-n \to -\infty$. As a consequence, there exists $c>0$ such that for all $n \geq 1$ and $z \geq m^{-n}$,
\[
  \left(s_{k-n}\right)^{-zm^{n}} \leq \exp\left( C z m^k\right).
\]
As $\lim_{z \to 0} zm^k= 0$, we conclude that there exists $C>0$ such that for all $n \geq 1$ and $z \geq m^{-n}$,
\[
  \P(Z_n \leq z m^n) \leq q + C f'(q)^\frac{-\log z}{\log m - \log f'(q)} + C z m^k = q + C z^{-\frac{\log f'(q)}{\log m - \log f'(q)}} = q + C z^\frac{\alpha}{\alpha + 1}.
\]

Similarly, if $b=1$, then $q=0$ and $f'(0) = \mu(1)$. We set $k = \frac{-\log z}{\log m}$. There exists $C>0$ such that for all $n \geq 1$ and $z \geq m^{-n}$, we have
\[
  \P(Z_n \leq z m^n) \leq C \mu(1)^{- \frac{\log z}{\log m}} \leq C z^{- \frac{\log \mu(1)}{\log m}}.
\]
Finally, if $b \geq 2$, we choose $k = - \frac{\log z}{\log m - \log b}$, there exists $c>0$ (small enough) such that
\[
  \P(Z_n \leq z m^n) \leq \exp\left[ - c z^{-\frac{\log b}{\log m - \log b}} \right],
\]
which ends the proof.
\end{proof}

Lemma \ref{lem:gw} is used to obtain a lower bound on the size of the population in a branching random walk above a given position.
\begin{lemma}
\label{lem:lbsizepop}
Under assumptions \eqref{eqn:reproduction}, there exist $a>0$ and $\rho>1$ such that a.s. for $n \geq 1$ large enough
\[
  \# \left\{ |u| = n : \forall j \leq n, V(u_j) \geq - n a \right\} \geq \rho^n.
\]
\end{lemma}

\begin{proof}
Using monotone convergence and \eqref{eqn:reproduction}, we have
\[\lim_{P \to +\infty} \lim_{a \to +\infty} \E\left[ P \wedge \sum_{|u|=1} \ind{V(u) \geq -a} \right] = \E\left[ \sum_{|u|=1} 1\right]> 1.\]
Hence there exists $a>0$ and $P \in \N$ such that $\rho_1 := \E\left[ P \wedge \sum_{|u|=1} \ind{V(u) \leq a} \right] >1$. We set $N  = P \wedge \sum_{|u|=1} \ind{V(u) \geq -a}$ --note that $\E(N) < +\infty$. Observe that a Galton-Watson process $Z$ with reproduction law $N$ can be coupled with the branching random walk $(\T,V)$ such that the following holds
\[
  \sum_{|u|=n} \ind{\forall j \leq n, V(u_j) \geq -ja} \geq Z_n.
\]
We write $p := \P(\forall n \in \N, Z_n>0)>0$ for the survival probability of this Galton-Watson process.

For $n \in \N$, we write $\tilde{Z}_n$ for the number of individuals with an infinite number of descendants. Conditionally on the survival of $Z$, the process $(\tilde{Z}_n, n \geq 0)$ is a supercritical Galton-Watson process that survives almost surely (see e.g. \cite{AtN04}). Applying Lemma \ref{lem:gw}, there exists $\rho>1$ such that
\[
  \P(\tilde{Z}_n \leq \rho^n) \leq \rho^{-n}.
\]
Applying the Borel-Cantelli lemma, a.s. for any $n \geq 1$ large enough $\tilde{Z}_n \geq \rho^n$.

We introduce a sequence of individuals $(u_n) \in \T^\N$ such that $|u_n|=n$, $u_0 = \emptyset$ and $u_{n+1}$ is the leftmost child of $u_n$, with ties broken uniformly at random. We write $q = \P(N \geq 2)$ for the probability that $u_n$ has at least two children, both of them above $-a$. We introduce the random time $T$ defined as the smallest $k \in \N$ such that the second leftmost child $v$ of $u_k$ is above $-a$, and the Galton-Watson process coupled with the branching random walk rooted at $v$ survives. We observe that $T$ is stochastically bounded by a geometric random variable with parameter $pq$, and that conditionally on $T$, the Galton-Watson tree that survives has the same law as $\tilde{Z}$.

Thanks to these observations, we note that $T<+\infty$ and $\inf_{j \leq T} V(u_j)>-\infty$. For any $n \geq 1$ large enough such that $T<n$ and $\inf_{j \leq T} V(u_j) \geq -na$ we have
\[
  \# \left\{ u \in \T : |u|=2n, \forall j \leq n, V(u_j) \geq -3na \right\} \geq \rho^n,
\]
as desired.
\end{proof}

\section{Branching random walk with a killing boundary at critical rate}
\label{sec:killing_critical}

In this section, we study the behaviour of a branching random walk on $\R$ in which individuals below a given barrier are killed. Given a continuous function $f \in \calC([0,1])$ such that $\limsup_{t \to 0} \frac{f_t}{t}<+\infty$ and $n \in \N$, for any $k \leq n$ every individual alive at generation $k$ below level $f_{k/n} n^{1/3}$ are removed, as well as all their descendants. Let $(\T,V)$ be a branching random walk, we denote by
\[
  \T^{(n)}_f = \left\{ u \in \T : |u| \leq n, \forall j \leq |u|, V(u_j) \geq n^{1/3} f(k/n) \right\},
\]
and note that $\T^{(n)}_f$ is a random tree. The process $(\T^{(n)}_f,V)$, called branching random walk with a killing boundary, has been introduced in \cite{AiJ11,Jaf12}, where a criterion for the survival of the process is obtained. In this section, we study the asymptotic behaviour of $(\T^{(n)}_f,V)$. More precisely, we compute the probability that $\T^{(n)}_f$ survives until time $n$, and provide bounds on the size of the population in $\T^{(n)}_f$ at any time $k \leq n$.

To obtain these estimates, we first find a function $g$ such that with high probability, no individual alive at generation $k$ in $\T^{(n)}_f$ is above $n^{1/3} g_{k/n}$. We compute in a second time the first and second moments of the number of individuals in $\T$ that stay at any time $k \leq n$ between $n^{1/3} f_{k/n}$ and $n^{1/3} g_{k/n}$.

With a careful choice of functions $f$ and $g$, one can compute the asymptotic behaviour of the consistent maximal displacement at time $n$, which is \cite[Theorem 1]{FaZ10} and \cite[Theorem 1.4]{FHS12}; or the asymptotic behaviour as $\epsilon \to 0$ of the probability there exists an individual in the branching random walk staying at any time $n \in \N$ above $-\epsilon n$, which is \cite[Theorem 1.2]{GHS11}. We present these results respectively in Theorem \ref{thm:fzfhs} and Theorem \ref{thm:ghs}, with weaker integrability conditions than in the seminal articles.

\subsection{Number of individuals in a given path}
\label{subsec:computations}

For any two continuous functions $f < g$, we denote by
\[
  H_t(f,g) = \frac{\pi^2 \sigma^2}{2} \int_0^t \frac{ds}{(g_s-f_s)^2}.
\]
For $n \geq 1$ and $k \leq n$, we write $I^{(n)}_k = [f_{k/n} n^{1/3}, g_{k/n} n^{1/3}]$. We compute in a first time the number of individuals in $\T^{(n)}_f$ that crosses for the first time at some time $k \leq n$ the frontier $g_{k/n}n^{1/3}$. We set
\[
  Y^{(n)}_{f,g} = \sum_{u \in \T^{(n)}_f} \ind{V(u) > g_{|u|/n}n^{1/3}} \ind{V(u_j) \leq g_{j/n} n^{1/3}, j < |u|}.
\]

\begin{lemma}
\label{lem:frontier}
Let $f \leq g$ such that $f_0 \leq 0 \leq g_0$. Under assumptions \eqref{eqn:reproduction} and \eqref{eqn:critical},
\begin{equation}
  \label{eqn:frontier}
  \limsup_{n \to +\infty} n^{-1/3} \log \E\left[ Y^{(n)}_{f,g}\right] \leq -\inf_{t \in [0,1]} g_t + H_t (f,g).
\end{equation}
\end{lemma}

\begin{proof}
Using Lemma \ref{lem:manytoone}, we have
\begin{align*}
  \E\left[ Y^{(n)}_{f,g} \right]
  &= \sum_{k=1}^n \E\left[ \sum_{|u| = k} \ind{V(u) \geq g_{k/n} n^{1/3}} \ind{V(u_j) \in I^{(n)}_j, j < k} \right]\\
  &= \sum_{k=1}^n \E\left[ e^{-S_k} \ind{S_k \geq g_{k/n} n^{1/3}} \ind{S_j \in I^{(n)}_j, j < k} \right]\\
  &\leq \sum_{k=1}^n e^{-n^{1/3} g_{k/n}} \P\left( S_j \in I^{(n)}j_j, j < k \right).
\end{align*}

Let $\delta > 0$, we set $I^{(n),\delta}_k = \left[ (f_{k/n}-\delta)n^{1/3}, (g_{k/n}+\delta) n^{1/3}\right]$. Let $A \in \N$, for $a \leq A$ we write $m_a = \floor{na/A}$ and $\underline{g}_{a,A} = \inf_{s \in [a/A,(a+1)/A]} g_s$. Applying the Markov property at time $m_a$, for any $k > m_a$, we have
\[
  e^{-n^{1/3} g_{k/n}} \P\left( S_j \in I^{(n)}j_j, j < k \right) \leq e^{-n^{1/3} \underline{g}_{a,A}} \P\left( S_j \in I^{(n)},\delta_j, j \leq m_a \right).
\]
Applying Theorem \ref{thm:mogulskii}, we have
\[
  \limsup_{n \to +\infty} n^{-1/3} \log \E\left[ Y^{(n)}_{f,g} \right] \leq \max_{a < A} -\underline{g}_{a,A} - H_{a/A}(f-\delta,g+\delta).
\]
Letting $\delta \to 0$ and $A \to +\infty$, we conclude that
\[
  \limsup_{n \to +\infty} n^{-1/3} \log \E\left[ Y^{(n)}_{f,g} \right] \leq \sup_{t \in [0,1]} - g_t - H_t(f,g).
\]
\end{proof}

Using this lemma, we note that if $\inf_{t \in [0,1]} g_t + H_t(f,g) \geq \delta$, then with high probability no individual in $\T^{(n)}_f$ crosses the curve $g_{./n} n^{1/3}$ with probability at least $1 - e^{-\delta n^{1/3}}$. In a second time, we take interest in the number of individuals that stays between $f_{./n}n^{1/3}$ and $g_{./n}n^{1/3}$. For any $f_1 \leq x < y \leq g_1$, we set
\[
  Z^{(n)}_{f,g}(x,y) = \sum_{|u|=n} \ind{V(u) \in [xn^{1/3},yn^{1/3}]} \ind{V(u_j) \in I^{(n)}_j, j \leq n}.
\] 

\begin{lemma}
\label{lem:meannumber}
Let $f<g$ be such that $\liminf_{t \to 0} \frac{g_t}{t}>-\infty$ and $\limsup_{t \to 0} \frac{f_t}{t}<+\infty$. Under assumptions \eqref{eqn:reproduction} and \eqref{eqn:critical}, we have
\[
  \lim_{n \to +\infty} n^{-1/3} \log \E\left(Z^{(n)}_{f,g}(x,y)\right) =  -(x + H_1(f,g)).
\]
\end{lemma}

\begin{proof}
Applying \eqref{eqn:manytoone}, we have
\[
  \E\left(Z^{(n)}_{f,g}(x,y)\right) = \E\left[ e^{-S_n} \ind{S_n \in [x n^{1/3}, y n^{1/3}]} \ind{S_j \in I^{(n)}_j , j \leq n} \right],
\]
which yields
\begin{equation}
  \label{eqn:uppermean}
  \E\left(Z^{(n)}_{f,g}(x,y)\right) \leq e^{-x n^{1/3}} \P\left( S_n \in [x n^{1/3}, y n^{1/3}], S_j \in I^{(n)}_j, j \leq n\right).
\end{equation}
Moreover, note that for any $\epsilon>0$, $Z^{(n)}_{f,g}(x,y) \geq Z^{(n)}_{f,g}(x,x+\epsilon)$, and we have
\begin{equation}
   \E(Z^{(n)}_{f,g}(x,y) \geq e^{-(x+\epsilon) n^{1/3}} \P\left( S_n \in [xn^{1/3}, (x+\epsilon)n^{1/3}], S_j \in I^{(n)}_j, j \leq n \right). \label{eqn:lowermean}
\end{equation}
As $f<g$, $\liminf_{t \to 0} \frac{g_t}{t}>-\infty$ and $\limsup_{t \to 0} \frac{f_t}{t}<+\infty$, either $f_0<0$ or $g_0>0$. Consequently, applying Corollary \ref{cor:mogulskii_modified}, for any $f_1 \leq x' < y' \leq g_1$ we have
\[
  \lim_{n \to +\infty} n^{-1/3} \log \P\left( S_n \in [x' n^{1/3}, y' n^{1/3}], S_j \in I^{(n)}_j, j \leq n \right) = -H_1(f,g).
\]
Therefore, \eqref{eqn:uppermean} yields
\[
  \limsup_{n \to +\infty} n^{-1/3} \log \E(Z^{(n)}_{f,g}(x,y)) \leq -x - H_1(f,g)
\]
and \eqref{eqn:lowermean} yields
\[
  \liminf_{n \to +\infty} n^{-1/3} \log \E(Z^{(n)}_{f,g}(x,y)) \geq -x -\epsilon - H_1(f,g).
\]
Letting $\epsilon \to 0$ concludes the proof.
\end{proof}

Lemma \ref{lem:meannumber} is used to bound from above the number of individuals in $\T^{(n)}_f$ who are at time $n$ in a given interval. To compute a lower bound we use a second moment concentration estimate. To successfully bound from above the second moment, we are led to restrict the set of individuals we consider to individuals with ``not too many siblings'' in the following sense. For $u \in \T$, we set 
\[
  \xi(u) = \log \left( 1 + \sum_{v \in \Omega(u)} e^{V(v)-V(u)} \right)
\]
where $\Omega(u)$ is the set of siblings of $u$, i.e., the set of children of the parent of $u$ except $u$ itself. For any $\delta > 0$ and $f_1 \leq x < y \leq g_1$, we write
\[
  \tilde{Z}^{(n)}_{f,g}(x,y,\delta) = \sum_{|u|=n} \ind{V(u) \in [x n^{1/3}, y n^{1/3}]} \ind{V(u_j) \in I^{(n)}_j, \xi(u_j) \leq \delta n^{1/3}, j \leq n},
\]
and note that for any $\delta >0$, $\tilde{Z}^{(n)}_{f,g}(x,y,\delta) \leq Z^{(n)}_{f,g}(x,y)$.

\begin{lemma}
\label{lem:secondorder}
Let $f< g$ be such that $\liminf_{t \to 0} \frac{g_t}{t}>-\infty$ and $\limsup_{t \to 0} \frac{f_t}{t}<+\infty$. Under assumptions \eqref{eqn:reproduction}, \eqref{eqn:critical} and \eqref{eqn:integrability}, for any $f_1 \leq x < y \leq g_1$ and $\delta > 0$ we have
\begin{equation}
  \label{eqn:secondordermean}
  \liminf_{n \to +\infty} n^{-1/3} \log \E(\tilde{Z}^{(n)}_{f,g}(x,y,\delta)) \geq  - (x + H_1(f,g)),
\end{equation}
\begin{equation}
  \label{eqn:secondordervariance}
  \limsup_{n \to +\infty} n^{-1/3} \log \E\left[ \left( \tilde{Z}^{(n)}_{f,g}(x,y,\delta)\right)^2 \right] \leq -2 (x + H_1(f,g)) + \delta + \sup_{t \in [0,1]} g_t + H_t(f,g).
\end{equation}
\end{lemma}

\begin{proof}
For any $\epsilon > 0$, applying Proposition \ref{pro:spinaldecomposition} we have
\begin{align*}
  &\E\left[ \tilde{Z}^{(n)}_{f,g}(x, y,\delta) \right]\\
  &= \bar{\E}\left[ \frac{1}{W_n} \sum_{|u|=n} \ind{V(u) \in [xn^{1/3}, yn^{1/3}]} \ind{V(u_j) \in I^{(n)}_j, j \leq n} \ind{\xi(u_j) \leq \delta n^{1/3}, j \leq n} \right]\\
  &\geq \hat{\E}\left[ e^{-V(w_n)} \ind{V(w_n) \in [xn^{1/3},(x+\epsilon)n^{1/3}]} \ind{V(w_j) \in I^{(n)}_j, \xi(w_j) \leq \delta n^{1/3}, j \leq n} \right]\\
  &\geq e^{-(x+\epsilon)n^{1/3}}\hat{\P} \left[ V(w_n) \in [xn^{1/3},(x+\epsilon)n^{1/3}], V(w_j) \in I^{(n)}_j, \xi(w_j) \leq \delta n^{1/3}, j \leq n \right].
\end{align*}
Setting $X=\xi(w_1)$, \eqref{eqn:integrability} implies $\hat{\E}(X^2)<+\infty$, thus $\lim_{z \to +\infty} z^2 \hat{\P}(X \geq z) = 0$. Applying Corollary \ref{cor:mogulskii_modified}, we obtain
\[
  \liminf_{n \to +\infty} n^{-1/3} \log \E\left[ \tilde{Z}^{(n)}_{f,g}(x,y,\delta) \right] \geq -(x+\epsilon) - H_1(f,g),
\]
and conclude the proof of \eqref{eqn:secondordermean} by letting $\epsilon \to 0$.

We now take care of the second moment. Using again Proposition \ref{pro:spinaldecomposition}, we have
\begin{align}
  &\E\left[(\tilde{Z}^{(n)}_{f,g}(x,y,\delta))^2\right]\nonumber\\
  = &\bar{\E}\left[ \frac{\tilde{Z}^{(n)}_{f,g}(x,y,\delta)}{W_n} \sum_{|u|=n} \ind{V(u) \in [x n^{1/3},y n^{1/3}]} \ind{V(u_j) \in I^{(n)}_j, j \leq n} \ind{\xi(u_j) \leq \delta n^{1/3}, j \leq n}\right]\nonumber\\
  \leq &\hat{\E}\left[ e^{-V(w_n)} Z^{(n)}_{f,g}(x,y) \ind{V(w_n) \in [x n^{1/3}, y n^{1/3}]} \ind{V(w_j) \in I^{(n)}_j, j \leq n} \ind{\xi(w_j) \leq \delta n^{1/3}, j \leq n} \right]\nonumber\\
  \leq &e^{-x n^{1/3}} \hat{\E}\left[ Z^{(n)}_{f,g}(x,y) \ind{V(w_n) \in [x n^{1/3}, y n^{1/3}]} \ind{V(w_j) \in I^{(n)}_j, j \leq n} \ind{\xi(w_j) \leq \delta n^{1/3}, j \leq n} \right].\label{eqn:z}
\end{align}
We decompose $Z^{(n)}_{f,g}(x,y)$ according to the generation at which individuals split with the spine. For $u,v \in \T$, we write $v \geq u$ if $v$ is a descendant of $u$. For $u \in \T$ we set
\[
 \Lambda(u) = \sum_{|v|=n, v \geq u} \ind{V(v) \in [x n^{1/3}, y n^{1/3}]} \ind{V(v_j) \in I^{(n)}_j, j \leq n}.
\]
We have
\[
  Z^{(n)}_{f,g}(x,y) = \ind{V(w_n) \in [x n^{1/3}, y n^{1/3}]} \ind{V(w_j) \in I^{(n)}_j, j \leq n} + \sum_{k=1}^{n} \sum_{u \in \Omega_k} \Lambda(u),
\]
where $\Omega_k = \Omega(w_k)$ is the set of siblings of $w_{k}$.

By definition of $\hat{\P}$, conditionally on $\hat{\calF}_k$ the subtree of the descendants of $u \in \Omega_k$ is distributed as a branching random walk starting from $V(u)$. For any $k \leq n$ and $u \in \Omega_k$, applying Lemma \ref{lem:manytoone} we have
\begin{align*}
  \E\left[ \left. \Lambda(u)\right|\hat{\calF}_k\right] 
  &= \ind{V(w_j) \in I^{(n)}_j, j \leq k-1} \E_{V(u)} \left[ \sum_{|v|=n-k} \ind{V(v) \in [x n^{1/3}, y n^{1/3}]} \ind{V(v_j) \in I^{(n)}_{k+j}, j \leq n-k} \right]\\
  &= \ind{V(w_j) \in I^{(n)}_j, j \leq k-1} e^{-V(u)} \E_{V(u)} \left[ e^{-S_{n-k}} \ind{S_{n-k} \in [x n^{1/3}, y n^{1/3}]} \ind{S_j \in I^{(n)}_{k+j}, j \leq n-k} \right]\\
  &\leq e^{V(w_k) - xn^{1/3}} e^{V(u) - V(w_k)} \P_{V(u)} \left[ S_j \in I^{(n)}_{k+j}, j \leq n-k \right].
\end{align*}
Thus, by definition of $\xi(w_{k-1})$,
\[
  \sum_{u \in \Omega_k} \E\left[ \left. \Lambda(u) \right| \hat{\calF}_k \right] \leq e^{V(w_k) - x n^{1/3}} e^{\xi(w_{k})} \sup_{z \in \R} \P_z \left[ S_j \in I^{(n)}_{k+j}, j \leq n-k \right].
\]

Let $A \in \N$. For any $a \leq A$ we write $m_a = \floor{na/A}$. For any $k \leq m_a$ and $z \in \R$, applying the Markov property at time $m_a-k$ we have
\[
  \P_z \left[ S_j \in I^{(n)}_{k+j}, j \leq n-k \right] \leq \sup_{z' \in \R} \P_{z'}\left[ S_j \in I^{(n)}_{m_a+j}, j \leq n-m_a \right].
\]
We write $\Psi^{(n)}_a = \sup_{z' \in \R} \P_{z'}\left[ S_j \in I^{(n)}_{m_a+j}, j \leq n-m_a \right]$. By Corollary \ref{cor:mogulskii}, we have
\[
  \limsup_{n \to +\infty} n^{-1/3} \log \Psi^{(n)}_a \leq - \left( H_1(f,g) - H_{a/A}(f,g) \right). 
\]
Moreover, \eqref{eqn:z} becomes
\begin{multline*}
  \E\left[ \left(\tilde{Z}^{(n)}_{f,g}(x,y) \right)^2 \right] \leq e^{-xn^{1/3}} \P( S_j \in I^{(n)}_j, j \leq n)\\
  + e^{-2 x n^{1/3}} \sum_{a=0}^{A-1}\Psi^{(n)}_{a+1} \sum_{k=m_a+1}^{m_{a+1}} \E\left[ e^{V(w_{k})} e^{\xi(w_{k})} \ind{V(w_j) \in I^{(n)}_j, \xi(w_j) \leq \delta n^{1/3}, j \leq n} \right].
\end{multline*}
We set $\bar{g}_{a,A} = \sup_{s \in [\frac{a}{A},\frac{a+1}{A}]} g_s$, we have
\[
  \E\left[ e^{V(w_{k})} e^{\xi(w_{k})} \ind{V(w_j) \in I^{(n)}_j, \xi(w_j) \leq \delta n^{1/3}, j \leq n} \right] \leq  e^{n^{1/3}(\bar{g}_{a,A}+\delta)} \P(S_j \in I^{(n)}_j, j \leq n).
\]
We apply Theorem \ref{thm:mogulskii} to obtain
\[
  \limsup_{n \to +\infty} n^{-1/3} \log \sum_{k=m_a+1}^{m_{a+1}} \E\left[ e^{V(w_{k})} \xi(w_{k}) \ind{V(w_j) \in I^{(n)}_j, \xi(w_j) \leq \delta n^{1/3}, j \leq n} \right] \leq \bar{g}_{a,A} + \delta - H_1(f,g).
\]

We conclude that
\[  \limsup_{n \to +\infty} n^{-1/3} \log \E\left[(\tilde{Z}^{(n)}_{f,g}(x,y))^2\right] \leq -(2x + H_1(f,g)) + \delta + \max_{a < A} \bar{g}_{a,A} + H_\frac{a+1}{A}(f,g).\]
Letting $A \to +\infty$ concludes the proof.
\end{proof}

A straightforward consequence of Lemma \ref{lem:secondorder} is a lower bound on the asymptotic behaviour of the probability for $Z^{(n)}_{f,g}$ to be positive.
\begin{corollary}
\label{cor:lowboundproba}
Under the assumptions of Lemma \ref{lem:secondorder}, we have
\[
  \liminf_{n \to +\infty} n^{-1/3} \log \P\left[ Z^{(n)}_{f,g}(x,y) \geq 1 \right] \geq -\sup_{t \in [0,1]} g_t + H_t(f,g).
\]
\end{corollary}

\begin{proof}
For any $\delta > 0$, we have $Z^{(n)}_{f,g}(x,y) \geq \tilde{Z}^{(n)}_{f,g}(x,y,\delta)$. As a consequence,
\[
  \P\left[ Z^{(n)}_{f,g}(x,y) \geq 1 \right]
  \geq \P\left[ \tilde{Z}^{(n)}_{f,g}(x,y,\delta) \geq 1 \right]
  \geq \frac{\E\left[ \tilde{Z}^{(n)}_{f,g}(x,y,\delta) \right]^2}{\E\left[ \tilde{Z}^{(n)}_{f,g}(x,y,\delta)^2 \right]}
\]
by the Cauchy-Schwarz inequality. Therefore using Lemma \ref{lem:secondorder} we have
\[
  \liminf_{n \to +\infty} n^{-1/3} \log \P\left[ Z^{(n)}_{f,g}(x,y) \geq 1 \right] \geq -\sup_{t \in [0,1]} g_t + H_t(f,g).
\]
\end{proof}

Another application of Lemma \ref{lem:secondorder} is a lower bound on the value of the sum of a large number of i.i.d. versions of $Z^{(n)}_{f,g}(x,y)$. This is useful observing that by Lemma \ref{lem:lbsizepop}, at time $k$ there exists with high probability at least $\rho^k$ individuals, each of which starting an independent branching random walk.

\begin{corollary}
\label{cor:sizepop}
Under the assumptions of Lemma \ref{lem:secondorder}, we set $(Z_{f,g}^{(n),j}(x,y), j \in \N)$ i.i.d. copies of $Z^{(n)}_{f,g}(x,y)$. Let $z > 0$, we write $p = \floor{e^{zn^{1/3}}}$. For any $\epsilon>0$, we have
\[
  \limsup_{n \to +\infty} n^{-1/3} \log \P\left[ \sum_{j=1}^p Z_{f,g}^{(n),j}(x,y) \leq \exp\left( n^{1/3}(z - x - H_1(f,g) - \epsilon\right) \right] \leq -z + \sup_{t \in [0,1]} g_t + H_t(f,g) .
\]
\end{corollary}

\begin{proof}
The proof is based on the following observation. Let $(X_j, j \in \N)$ be i.i.d. random variables with finite variance. Using the Bienaymé-Chebychev inequality, we have
\begin{multline}
  \label{eqn:concentration}
  \P\left( \sum_{j=1}^p X_j \leq \frac{1}{2}\E\left( \sum_{j=1}^p X_j \right) \right) \leq \P\left( \left|\sum_{j=1}^p X_j - p \E(X_1) \right| \geq p\E(X_1)/2 \right)\\
  \leq 4 \frac{\Var\left( \sum_{j=1}^p X_j \right)}{p^2 \E(X_1)^2} \leq 4 \frac{\Var(X_1)}{p \E(X_1)} \leq 4 \frac{\E(X_1^2)}{p\E(X_1)^2}.
\end{multline}

Let $\delta > 0$, as $Z^{(n)}_{f,g}(x,y) \geq \tilde{Z}^{(n)}_{f,g}(x,y,\delta)$, we have
\begin{multline*}
  \P\left[ \sum_{j=1}^p Z_{f,g}^{(n),j}(x,y) \leq \exp\left( n^{1/3}(z - x - H_1(f,g) - \epsilon)\right) \right] \\ \leq \P\left[ \sum_{j=1}^p \tilde{Z}_{f,g}^{(n),j}(x,y,\delta) \leq \exp\left( n^{1/3}(z - x - H_1(f,g) - \epsilon\right) \right],
\end{multline*}
where $(\tilde{Z}^{(n),j}_{f,g}(x,y,\delta), j \in \N)$ is a sequence of i.i.d. copies of $\tilde{Z}_{f,g}^{(n),j}(x,y,\delta)$. By Lemma \ref{lem:secondorder},
\[
  \liminf_{n \to +\infty} n^{-1/3} \log \E\left( \tilde{Z}^{(n)}_{f,g}(x,y,\delta) \right) \geq - \left( x + H_1(f,g) \right),
\]
thus, for any $\epsilon>0$, for any $n \geq 1$ large enough we have
\[
  \E\left( \tilde{Z}^{(n)}_{f,g}(x,y,\delta) \right)/2 \geq e^{- n^{1/3}(x+H_1(f,g) + \epsilon)}.
\]
Therefore, using again Lemma \ref{lem:secondorder} and \eqref{eqn:concentration}, we have
\begin{multline*}
  \limsup_{n \to +\infty} n^{-1/3} \log \P\left[ \sum_{j=1}^p \tilde{Z}_{f,g}^{(n),j}(x,y,\delta) \leq \exp\left( n^{1/3}(z - x-H_1(f,g) - \epsilon \right) \right] \\
   \leq -z +\delta + \sup_{t \in [0,1]} g_t + H_t(f,g).
\end{multline*}
Consequently, letting $\delta \to 0$ we have
\[
  \limsup_{n \to +\infty} n^{-1/3} \log \P\left[ \sum_{j=1}^p Z_{f,g}^{(n),j}(x,y) \leq \exp\left( n^{1/3}(z - x-H_1(f,g) - \epsilon \right) \right] \leq -z + \sup_{t \in [0,1]} g_t + H_t(f,g).
\]
\end{proof}

\subsection{Asymptotic behaviour of the branching random walk with a killing boundary}
\label{subsec:behaviour}

The results of Section \ref{subsec:computations}, in particular Lemma \ref{lem:frontier} and Corollaries \ref{cor:lowboundproba} and \ref{cor:sizepop}, emphasize the importance of the functions $g$ verifying
\begin{equation}
  \label{eqn:defineg}
  \forall t \in [0,1], g_t = g_0 - H_t(f,g) > f_t,
\end{equation}
in the study of $\T^{(n)}_f$. For such a function, the estimates of Lemmas \ref{lem:frontier}, \ref{lem:meannumber} and \ref{lem:secondorder} are tight. They enable to precisely study the asymptotic behaviour of $\T^{(n)}_f$.

\begin{theorem}
\label{thm:Tfbehaviour}
We consider a branching random walk $(\T,V)$ satisfying \eqref{eqn:reproduction}, \eqref{eqn:critical} and \eqref{eqn:integrability}. Let $f \in \calC([0,1])$ be such that $f_0<0$. If there exists a continuous function $g$ such that
\[
  g_0=0 , \quad \forall t \in [0,1], g_t = -\frac{\pi^2 \sigma^2}{2} \int_0^t \frac{ds}{(g_s-f_s)^2} \quad \mathrm{and} \quad \forall t \in [0,1], g_t > f_t,
\]
then almost surely for $n \geq 1$ large enough, $\{u \in \T^{(n)}_f : |u|=n\} \neq \emptyset$ and
\begin{multline}
  \label{eqn:boundsfortf}
  \lim_{n \to +\infty} \frac{1}{n^{1/3}} \# \left\{ u \in \T^{(n)}_f : |u| = n \right\} = g_1 -f_1,\\
  \lim_{n \to +\infty} \frac{1}{n^{1/3}}\min_{u \in \T^{(n)}_f , |u|=n} V(u)= f_1 \quad \mathrm{and} \quad \lim_{n \to +\infty} \frac{1}{n^{1/3}}\max_{u \in \T^{(n)}_f , |u|=n} V(u) = g_1 \quad \mathrm{a.s.} 
\end{multline}
Otherwise, writing
\begin{equation}
  \label{eqn:deflambda}
  \lambda = \inf\left\{ g_0 , g \in \calC([0,1]) : \forall t \in [0,1], g_t = g_0 - \frac{\pi^2 \sigma^2}{2}\int_0^t \frac{ds}{(g_s-f_s)^2} > f_t \right\},
\end{equation}
then
\begin{equation}
  \label{eqn:finlambda}
  \lim_{n \to +\infty} n^{-1/3} \log \P\left(\left\{u \in \T^{(n)}_f : |u|=n\right\} \neq \emptyset \right) = -\lambda.
\end{equation}
\end{theorem}

\begin{proof}
We study the solutions of the differential equation \eqref{eqn:defineg}. As $(t,x) \mapsto -\frac{\pi^2 \sigma^2}{2(x-f_t)^2}$ is locally Lipschitz on $\{(t,x) \in [0,1] \times \R : x > f_t\}$, the Cauchy-Lipschitz theorem implies that for any $x > f_0$, there exists a unique continuous function $g^x$ defined on the maximal interval $[0,t_x]$ such that $g^x_0 = x$, either $t_x =1$ or $g_{t_x}=f_{t_x}$, and for any $t< t_x$
\[
  g^x_t = x - \frac{\pi^2 \sigma^2}{2} \int_0^t \frac{ds}{(g^x_s - f_s)^2}.
\]
Moreover, we observe that $t_x$ is increasing with respect to $x$ and $g^x_t$ is decreasing in $t$ and increasing in $x$ on $\{ (t,x) \in [0,1] \times (f_0,+\infty) : t \leq t_x\}$. With these notations, we have
\[
  \lambda = \inf\left\{ x >f_0 : t_x = 1 \right\}.
\]
As $\lim_{x \to +\infty} \sup_{t \in [0,1]} \frac{\pi^2 \sigma^2}{2(x-f_t)^2} = 0$, there exists $x>0$ large enough such that $t_x=1$. This implies $\lambda < +\infty$.

We note that for any $x > 0$ such that $g^x > f$ on $[0,1]$, applying Corollary \ref{cor:lowboundproba} we obtain
\begin{align*}
  \liminf_{n \to +\infty} n^{-1/3} \log \P\left[ \{u \in \T^{(n)}_f : |u|=n\} \neq \emptyset \right]
  &\geq \liminf_{n \to +\infty} n^{-1/3} \log \P\left[ Z^{(n)}_{f,g^x}(f_1),g^x_1) \geq 1 \right]\\
  &\geq - x.
\end{align*}
Therefore, we have $\liminf_{n \to +\infty} n^{-1/3} \log \P\left[ \{u \in \T^{(n)}_f : |u|=n\} \neq \emptyset \right] \geq - \min(\lambda,0)$.

If $\lambda \geq 0$, writing $t = t_\lambda$, we use the fact that at some time before $t_\lambda$ every individual in $\T^{(n)}_f$ crosses $n^{1/3} g_{./n}$ before time $tn$, thus
\[
  \P\left( \exists |u| = n : u \in \T^{(n)}_f \right) \leq \P\left( \exists u \in \T^{(n)}_f : V(u) \geq n^{1/3} g_{|u|/n} \right).
\]
We set $f^{(1)}_s = f_{st}/t^{1/3}$ and $g^{(1)}_s = g^\lambda_{st}/t^{1/3}$. Applying Lemma \ref{lem:frontier}, and writing $m=\floor{tn}$ we have
\[
  \limsup_{n \to +\infty} n^{-1/3} \log \E\left(Y^{(m)}_{f^{(1)},g^{(1)}}\right) \leq - \lambda,
\]
which by Markov inequality yields
\[
  \limsup_{n \to +\infty} n^{-1/3} \log \P\left( u \in \T_f : |u| \leq tn, V(u) \geq n^{1/3} g_{|u|/n} \right) \leq - \lambda,
\]
concluding the proof of \eqref{eqn:finlambda}.

We now assume $\lambda < 0$, or equivalently $g^0 > f$. Applying Lemma \ref{lem:frontier}, for any $\epsilon>0$ we have
\[
  \limsup_{n \to +\infty} n^{-1/3} \log \P\left( \exists u \in \T^{(n)}_f : V(u) \geq n^{1/3} g^\epsilon_{|u|/n} \right) \leq -\inf_{t \in [0,1]} g^\epsilon_t + H_t(f,g^\epsilon) = - \epsilon.
\]
By the Borel-Cantelli lemma, almost surely for any $n \geq 1$ large enough, we have
\begin{equation}
  \label{eqn:frontiereps}
  \left\{ u \in \T^{(n)}_f : V(u) \geq n^{1/3} g^\epsilon_{|u|/n} \right\} = \emptyset.
\end{equation}
In particular, letting $\epsilon \to 0$ we have
\[
  \limsup_{n \to +\infty} \frac{1}{n^{1/3}} \max_{u \in \T^{(n)}_f , |u|=n} V(u) = g_1 \quad \mathrm{a.s.}
\]
Moreover, by Lemma \ref{lem:meannumber} we have
\[
  \limsup_{n \to +\infty} \frac{1}{n^{1/3}} \log \E\left[ Z^{(n)}_{f,g^\epsilon}(f_1,g_1^\epsilon) \right] \leq - (f_1 + H_1(f,g^\epsilon)) = g^\epsilon_1-f_1 - \epsilon. 
\]
Thus, by the Markov inequality and the Borel-Cantelli Lemma
\[
  \limsup_{n\to +\infty} n^{-1/3} \log Z^{(n)}_{f,g^\epsilon}(f_1,g^\epsilon_1) \leq g^\epsilon_1 - f_1.
\]
Mixing with \eqref{eqn:frontiereps} and letting $\epsilon \to 0$, we conclude
\[
  \limsup_{n \to +\infty} \frac{1}{n^{1/3}} \log \# \left\{ u \in \T^{(n)}_f : |u| = n \right\} \leq g_1 - f_1.
\]

To obtain the other bounds of \eqref{eqn:boundsfortf}, we apply Lemma \ref{lem:lbsizepop}. For any $\epsilon>0$ there exists $\rho>1$ and $\delta>0$  such that almost surely for any $n \geq 1$ large enough,
\[
  \# \left\{ u \in \T^{(n)}_f : |u| = \floor{\delta n^{1/3}} \text{ and } V(u) \in [ -\epsilon n^{1/3}, \epsilon n^{1/3} ] \right\} \geq \rho^{\delta n^{1/3}}.
\]
We write $S_n$ this event. On $S_n$, each of these $\rho^{\delta n^{1/3}}$ individuals starts an independent branching random walk from some point in $[-\epsilon n^{1/3},\epsilon n^{1/3}]$ with a killing boundary $n^{1/3} f_{./n}$. For $\epsilon$ small enough, we use Corollary \ref{cor:sizepop} to bound from below the number of descendants that stay between $f+2\epsilon$ and $g^{-2\epsilon}+2\epsilon$. We have
\begin{multline*}
  \limsup_{n \to +\infty} n^{-1/3} \log \P\left[ \left.\# \left\{ u \in \T^{(n)}_f : |u| = n \right\} \leq e^{n^{1/3}(g^{-2\epsilon}_1-f_1)} \right| S_n \right]\\
  \leq - \eta + \sup_{t \in [0,1]} g^{-2\epsilon}_t + 2 \epsilon + H_t(f+2\epsilon, g^{-2\epsilon} + 2 \epsilon) = -\eta.
\end{multline*}
Using again the Borel-Cantelli lemma, we obtain
\[
  \liminf_{n \to +\infty} n^{-1/3} \log \# \left\{ u \in \T^{(n)}_f : |u| = n \right\}  \geq g^{-2\epsilon}_1-f_1 \quad \mathrm{a.s.}
\]
Consequently, letting $\epsilon \to 0$ we conclude
\[
  \lim_{n \to +\infty} n^{-1/3} \log \# \left\{ u \in \T^{(n)}_f : |u| = n \right\} = g^0_1-f_1 \quad \mathrm{a.s.}
\]
In particular, almost surely for $n \geq 1$ large enough, $\T^{(n)}_f$ survives until time $n$, which is enough to prove
\[
  \liminf_{n \to +\infty} \frac{1}{n^{1/3}}\min_{u \in \T^{(n)}_f , |u|=n} V(u) \geq f_1 \quad \mathrm{a.s.}
\]

We observe by Corollary \ref{cor:lowboundproba} that for any $\epsilon>0$ small enough, for any $f_1 +2\epsilon < x < y < g^{-2\epsilon}_1 + 2\epsilon$ we have
\[
  \liminf_{n \to +\infty} n^{-1/3} \log \P\left( Z^{(n)}_{f+2\epsilon,g^{-2\epsilon}+2\epsilon}(x,y) > 0 \right) \geq 0.
\]
Therefore, for any $f_1 < x < y < g_1$, for any $\epsilon>0$ small enough we have
\[
  \P\left( \left. Z^{(n)}_{f,g}(x,y) = 0 \right| S_n \right) = \left( 1 - e^{o(n^{1/3})} \right)^{e^\eta n^{1/3}}.
\]
We conclude that for any $\zeta > 0$ small enough,
\[
  \liminf_{n \to +\infty} n^{-1/3} \log \left( - \log \P\left(Z^{(n)}_{f,g}(f_1 + \zeta, f_1 + 2 \zeta) = 0 \right) \right) > 0
\]
as well as
\[
  \liminf_{n \to +\infty} n^{-1/3} \log \left( - \log \P\left(Z^{(n)}_{f,g}(g_1 - 2 \zeta, g_1 - \zeta) = 0 \right) \right) > 0.
\]
Using once again the Borel-Cantelli lemma, we obtain respectively
\[
  \limsup_{n \to +\infty} \frac{1}{n^{1/3}}\min_{u \in \T^{(n)}_f , |u|=n} V(u) \leq f_1 \quad \mathrm{a.s.}
\]
and
\[
  \liminf_{n \to +\infty} \frac{1}{n^{1/3}}\max_{u \in \T^{(n)}_f , |u|=n} V(u) \geq g_1^0 \quad \mathrm{a.s.}
\]
which concludes the proof.
\end{proof}

\subsection{Applications}
\label{subsec:applications}

Using the results developed in this section, we deduce the asymptotic behaviour of the consistent maximal displacement at time $n$ of the branching random walk.
\begin{theorem}[Consistent maximal displacement of the branching random walk, \cite{FaZ10,FHS12}]
\label{thm:fzfhs}
We consider a branching random walk $(\T,V)$ satisfying \eqref{eqn:reproduction}, \eqref{eqn:critical} and \eqref{eqn:integrability}. We have
\[
  \lim_{n \to +\infty} \frac{\max_{|u|=n} \min_{k \leq n} V(u_k)}{n^{1/3}} = - \left( \frac{3\pi^2\sigma^2}{2} \right)^{1/3}.
\]
\end{theorem}

\begin{proof}
To prove this result, we only have to show that for any $\delta > 0$, almost surely for $n \geq 1$ large enough we have
\[
  \left\{ u \in \T^{(n)}_{\left(-\frac{3\pi^2\sigma^2}{2} \right)^{1/3}+\delta} : |u|=n \right\} = \emptyset \quad \mathrm{and} \quad \left\{ u \in \T^{(n)}_{\left(-\frac{3\pi^2\sigma^2}{2} \right)^{1/3}-\delta} : |u|=n \right\} \neq \emptyset.
\]
We solve for $x < 0$ the differential equation
\[
  g_t = - \frac{\pi^2 \sigma^2}{2} \int_0^t \frac{ds}{(g_s - x)^2},
\]
which is $g_t = x + \left( -x^3 - \frac{3\pi^2 \sigma^2}{2} t \right)^{1/3}$ for $t < \frac{-2x^3}{3 \pi^2 \sigma^2}$. By Theorem \ref{thm:Tfbehaviour}, for any $x >  -\left( \frac{3\pi^2\sigma^2}{2} \right)^{1/3}$, almost surely for any $n \geq 1$ large enough the tree $\T^{(n)}_x$ gets extinct before time $n$. For any $x < -\left( \frac{3\pi^2\sigma^2}{2} \right)^{1/3}$, almost surely for $n \geq 1$ large enough the tree $\T^{(n)}_x$ survives until time $n$.
\end{proof}

Similarly, we provide the asymptotic behaviour, as $\epsilon \to 0$ of the probability of survival of a branching random walk with a killing boundary of slope $-\epsilon$.
\begin{theorem}[Survival probability in the killed branching random walk \cite{GHS11}]
\label{thm:ghs}
Let $(\T,V)$ be a branching random walk satisfying \eqref{eqn:reproduction}, \eqref{eqn:critical} and \eqref{eqn:integrability}. We have
\[
  \lim_{\epsilon \to 0} \epsilon^{1/2} \log \P\left( \forall n \in \N, \exists |u|=n : V(u_j) \geq -\epsilon j, j \leq n  \right) = -\frac{\pi \sigma}{2^{1/2}}.
\]
\end{theorem}

\begin{proof}
For any $\epsilon>0$ and $n \in \N$, we set $\rho(n,\epsilon) = \P\left( \exists |u|=n : V(u_j) \geq - \epsilon j, j \leq n \right)$ and
\[
  \rho(\epsilon) = \lim_{n \to +\infty} \rho(n,\epsilon) = \P\left( \forall n \in \N, \exists |u|=n : V(u_j) \geq -\epsilon j, j \leq n  \right).
\]
In a first time, we prove that for any $\theta > 0$,  we have
\begin{equation}
  \label{eqn:lem46ghs}
  -\frac{\pi \sigma}{(2\theta)^{1/2}} \leq \liminf_{n \to +\infty} n^{-1/3} \log \rho\left( n, \theta n^{-2/3} \right) \leq \limsup_{n \to +\infty} n^{-1/3} \log \rho\left( n,\theta n^{-2/3} \right) \leq \Phi^{-1}(\theta),
\end{equation}
where $\Phi : \lambda \mapsto \frac{\pi^2 \sigma^2}{2\lambda^2} - \frac{\lambda}{3}$.

Applying Lemma \ref{lem:frontier} with functions $f : t \mapsto -\theta t$ and $g : t \mapsto \lambda (1 -t)^{1/3} - \theta t$ we prove the upper bound of \eqref{eqn:lem46ghs}. Using the fact that an individual staying above $f^{(n)}$ until time $n$ crosses $g^{(n)}$ at some time $k \leq n$, the Markov inequality implies
\begin{align*}
  \limsup_{n \to +\infty} n^{-1/3} \log \rho(n,\theta n^{-2/3})&\leq \limsup_{n \to +\infty} n^{-1/3} \log \E(Y^{(n)}_{f,g})\\
  &\leq -\inf_{t \in [0,1]} g_t + H_t(f,g)\\
  &\leq -\inf_{t \in [0,1]} \lambda (1 -t)^{1/3} - \theta t + \frac{\pi^2 \sigma^2}{2} \int_0^t \frac{ds}{(\lambda (1 -s)^{1/3})^2}\\
  &\leq -\inf_{t \in [0,1]} \lambda - \theta t + 3 \Phi(\lambda) \left[ 1 - (1 - t)^{1/3} \right].
\end{align*}
We observe that $t \mapsto 1 - (1 - t)^{1/3}$ is a convex function on $[0,1]$, with derivative $1/3$ at $t=0$. Thus, for any $\lambda > 0$ such that $\Phi(\lambda)>0$, for all $t \in [0,1]$, $3\Phi(\lambda) \left[ 1 - (1 - t)^{1/3} \right] \geq \Phi(\lambda) t$. We conclude that for any $\lambda > 0$ such that $\Phi(\lambda) \geq \theta > 0$, we have
\[
  \limsup_{n \to +\infty} n^{-1/3} \log \rho(n,\theta n^{-2/3}) \leq -\lambda.
\]
With $\lambda = \Phi^{-1}(\theta)$, we conclude the proof of the upper bound of \eqref{eqn:lem46ghs}. We now observe that for any $\epsilon>0$, we have $\rho(\epsilon) \leq \rho(n,\epsilon)$. Setting $n = \floor{(\theta/\epsilon)^{3/2}}$, for any $\theta > 0$ we have
\[
  \limsup_{\epsilon \to 0} \epsilon^{1/2} \log \rho(\epsilon) \leq \limsup_{\epsilon \to 0} \epsilon^{1/2} \log \rho(n,\epsilon) \leq -\theta^{1/2} \Phi^{-1}(\theta).
\]
We note that $\lim_{\theta \to +\infty} \theta^{1/2} \Phi^{-1}(\theta) = \lim_{\lambda \to 0} \lambda \Phi(\lambda)^{1/2} = \frac{\pi \sigma}{2^{1/2}}$, which concludes the proof of the upper bound in Theorem \ref{thm:ghs}.

To prove the lower bound in \eqref{eqn:lem46ghs}, we apply Corollary \ref{cor:lowboundproba} to functions $f : t \mapsto - \theta t$ and $g : t \mapsto \lambda - \theta t$. We have
\begin{align*}
  \liminf_{n \to +\infty} n^{-1/3} \log \rho(n,\theta n^{-2/3}) &\geq \liminf_{n \to +\infty} n^{-1/3} \log \P\left( Z^{(n)}_{f,g}(f_1,g_1) \geq 1\right)\\
  &\geq -\sup_{t \in [0,1]} \lambda - \theta t + \frac{\pi^2\sigma^2}{2\lambda^2} t. 
\end{align*}
Choosing $\lambda = \frac{\pi \sigma}{(2 \theta)^{1/2}}$, we obtain
\[
  \liminf_{n \to +\infty} n^{-1/3} \log \rho(n,\theta n^{-2/3}) \geq - \frac{\pi \sigma}{(2\theta)^{1/2}},
\]
proving the lower bound of \eqref{eqn:lem46ghs}. To extend this lower bound into the lower bound in Theorem \ref{thm:ghs} needs more care than the upper bound. First, we observe that this equation implies that for any $\theta > 0$,
\[
  \liminf_{n \to +\infty} n^{-1/3} \log \rho(\theta^{3/2} n, n^{-2/3}) \geq - \frac{\pi \sigma}{2^{1/2}}.
\]

By \eqref{eqn:reproduction}, there exist $a>0$ and $P \in \N$ such that $\E\left( \left(\sum_{|u|=1} \ind{V(u) \geq -a}\right)\wedge P \right)>1$. Consequently, there exists $\rho>1$ and a random variable $W$ positive with positive probability such that
\[
  \liminf_{n \to +\infty} \frac{\# \left\{ |u| = n : \forall j \leq n, V(u_j) \geq - aj \right\}}{\rho^n} \geq W \quad \mathrm{a.s.}
\]
We conclude there exists $a>0$, $r>0$ and $\rho>1$ such that
\[
  \inf_{n \in \N} \P\left(\# \left\{ |u| = n : \forall j \leq n, V(u_j) \geq - aj \right\} \geq \rho^n \right) \geq r.
\]
With these notations, we observe that for any $\theta > 0$, $\epsilon>0$, $\delta > 0$ and $n \in \N$, we have
\[
  \P\left( \# \left\{ |u| = (\theta + \delta) n : \forall j \leq n, V(u_j) \geq - \left( \frac{\theta \epsilon + \delta a}{\theta + \delta} \right) j \right\} \geq \rho^{\delta n} \right) \geq r\rho\left(\theta n, \epsilon\right).
\]
Given $\lambda > \frac{\pi \sigma}{2^{1/2}}$ and $\theta >0$, we set $\epsilon>0$ small enough such that
\[\epsilon^{1/2} \log \rho\left(\ceil{2 \theta^2 \epsilon^{-3/2}},\epsilon\right)>-\lambda.\]
We write $\delta = \frac{\theta\epsilon}{a-2\epsilon}$ and $n = \floor{(\theta+\delta) \epsilon^{-3/2}}$, choosing $\epsilon>0$ small enough such that $\delta < \theta$. We have
\[
  \P\left( \# \left\{ |u| = n : \forall j \leq n, V(u_j) \geq - 2 \epsilon j \right\} \geq \rho^{\delta n} \right) \geq re^{-\lambda \epsilon^{-1/2}},
\]
We construct a Galton-Watson process $(G_p(\epsilon), p \geq 0)$ based on the branching random walk $(\T,V)$ such that
\[
  G_p(\epsilon) = \#\left\{ |u|=pn : \forall j \leq pn, V(u_j) \geq - 2 \epsilon j \right\}.
\]
We observe that $G(\epsilon)$ stochastically dominates a Galton-Watson process $\tilde{G}(\epsilon)$, in which individuals make $N_\epsilon = \floor{\rho^{\delta n}}$ children with probability $p_\epsilon = r e^{-\lambda \epsilon^{-1/2}}$ and none with probability $1-p$. As $\epsilon \to 0$ we have
\[
  \lim_{\epsilon \to 0} \epsilon^{1/2} \log (p_\epsilon N_\epsilon) =  -\lambda + \frac{\theta^2 \log \rho}{a},
\]
which is positive choosing some $\theta > 0$ large enough. With this choice of $\theta$, for any $\epsilon>0$ small enough $p_\epsilon N_\epsilon > 2$. Consequently $q_\epsilon$ the probability of survival of $\tilde{G}(\epsilon)$ is positive for any $\epsilon>0$ small enough. Moreover, we have $\rho(2\epsilon) \geq q_\epsilon$.

We introduce $f_\epsilon : s \mapsto \E(s^{\tilde{G}(\epsilon)})$ which is a convex function verifying
\[f_\epsilon(1)=1 \text{ and } f_\epsilon(1 - q_\epsilon)=1-q_\epsilon.\]
For any $h>0$, for any $\epsilon>0$ small enough
\[
  f_\epsilon(1 - hp_\epsilon) = 1 - p_\epsilon + p_\epsilon (1 - hp_\epsilon)^{N_\epsilon} \leq 1 - p_\epsilon + p_\epsilon\exp(-hp_\epsilon N_\epsilon) \leq 1 - p_\epsilon + p_\epsilon e^{-2h}.
\]
Choosing $h>0$ small enough, for any $\epsilon>0$ small enough we have $f_\epsilon(1-hp)< 1 - hp$. This proves that $q_\epsilon>hp_\epsilon$, leading to
\[
  \liminf_{\epsilon \to 0} \epsilon^{1/2} \log \rho(\epsilon) \geq \liminf_{\epsilon \to 0} \epsilon^{1/2} \log p_\epsilon \geq - \lambda.
\]
Letting $\lambda \to - \frac{\pi \sigma}{2^{1/2}}$ concludes the proof.
\end{proof}

\section{Branching random walk with selection}
\label{sec:selection}

In this section, we consider a branching random walk on $\R$ in which at each generation only the rightmost individuals live. Given a positive continuous function $h$, at any time $k \leq n$ only the $\floor{e^{n^{1/3}h_{k/n}}}$ rightmost individuals remain alive. The process is constructed as follows. Let $((\T^p, V^p), p \in \N)$ be an i.i.d. sequence of independent branching random walks, for any $n \in \N$ we write $\calT_{(n)}$ for the disjoint union of $\T^p$ for $p \leq n$, and $V : u \in \calT_{(n)} \mapsto V^p(u)$ if $u \in \T^p$. We rank individuals at a given generation according to their position, from highest to lowest, breaking ties uniformly at random. For any $u \in \calT_{(n)}$, we write $N_{(n)}(u)$ for the ranking of $u$ in the $|u|^\mathrm{th}$ generation.

Let $h$ be a positive continuous function on $[0,1]$, we write $q = \floor{e^{h_0 n^{1/3}}}$ and
\[
  \T^h_{(n)} = \left\{ u \in \calT_{(q)} : |u| \leq n, \forall j \leq |u|, \log N_{(q)}(u_j) \leq n^{1/3} h_{j/n} \right\}.
\]
The process $(\T^h_{(n)},V)$ is a branching random walk with selection of the $e^{n^{1/3}h_{\cdot}}$ rightmost individuals. We write
\[
  M_n^h = \max_{u \in \T^h_{(n)}, |u| = n} V(u) \quad \mathrm{and} \quad m_n^h = \min_{u \in \T^h_{(n)}, |u|=n} V(u).
\]
We study $(\T^h_{(n)},V)$ by comparing it with $q$ independent branching random walks with a killing frontier $f$, choosing $f$ in a way that
\[
  \log \# \left\{ u \in \T^{(n)}_f : |u| = \floor{tn} \right\} \approx n^{1/3} (h_t-h_0).
\]
Using Lemmas \ref{lem:frontier} and \ref{lem:meannumber}, we choose functions $(f,g)$ verifying
\[ \forall t \in [0,1],
  \begin{cases}
    g_t + \frac{\pi^2\sigma^2}{2} \int_0^t \frac{ds}{(g_s - f_s)^2} = h_0\\
    f_t + \frac{\pi^2 \sigma^2}{2} \int_0^t \frac{ds}{(g_s - f_s)^2} = h_0-h_t.
  \end{cases}
\]
which solution is
\begin{equation}
  \label{eqn:fandg}
  f : t \in [0,1] \mapsto h_0 -h_t - \frac{\pi^2 \sigma^2}{2} \int_0^t \frac{ds}{h_s^2} \quad \mathrm{and} \quad g : t \in [0,1] \mapsto h_0 - \frac{\pi^2 \sigma^2}{2} \int_0^t \frac{ds}{h_s^2}.
\end{equation}

To compare branching random walk with selection and branching random walks with killing boundary, we couple them in a fashion preserving a certain partial order, that we describe now. Let $\mu,\nu$ be two Radon measures on $\R$, we write
\[
  \mu \preccurlyeq \nu \iff \forall x \in \R, \mu((x,+\infty)) \leq \nu((x,+\infty)).
\]
The relation $\preccurlyeq$ forms a partial order on the set of Radon measures, that can be used to rank populations, representing an individual by a Dirac mass at its position.

A branching-selection process is defined as follows. Given $\phi : \Z_+ \to \N$ a process adapted to the filtration of $\calT(\phi_0)$, we denote by
\[
  \T^\phi = \left\{ u \in \calT_{(\phi_0)} : \forall j \leq |u|, N_{(\phi_0)}(u_j) \leq \phi_j \right\}.
\]
Let $(x_1, \ldots x_{\phi_0}) \in \R^{\phi_0}$, we write $V : u \in \T^\phi \mapsto x_p + V^p(u)$ if $u \in \T^p$. The process $(\T^\phi,V)$ is a branching-selection process with $\phi(n)$ individuals at generation $n$ and initial positions $(x_1,\ldots x_{\phi_0})$. Note that both $\T^h_{(n)}$ and $\T_f^{(n)}$ can be described as branching-selection processes. We prove there exists a coupling between branching-selection processes preserving partial order $\preccurlyeq$. Note this lemma is essentially an adaptation of \cite[Corollary 2]{BeG10}.
\begin{lemma}
\label{lem:coupling}
Let $\phi$ and $\psi$ be two adapted processes, on the event
\[
  \left\{\sum_{\substack{u \in \T^\phi\\ |u|=0}} \delta_{V(u)} \preccurlyeq \sum_{\substack{u \in \T^\psi\\ |u|=0}} \delta_{V(u)} \quad \mathrm{and} \quad \forall j \leq n, \phi_j \leq \psi_j\right\},
\]
we have $\sum_{u \in \T^\phi, |u|=n} \delta_{V(u)} \preccurlyeq \sum_{u \in \T^\psi, |u|=n} \delta_{V(u)}$.
\end{lemma}

\begin{proof}
The lemma is a direct consequence of the following observation. Given $m \leq n$, $x \in \R^m$ and $y \in \R^n$ such that $\sum_{j=1}^m \delta_{x_j} \preccurlyeq \sum_{j=1}^n \delta_{y_j}$ and $(z^j_i, j \leq n, i \in \N)$, we have
\[
  \sum_{j=1}^m \sum_{i=1}^{+\infty} \delta_{x_j + z^j_i} \preccurlyeq \sum_{j=1}^n \sum_{i=1}^{+\infty} \delta_{y_j + z^j_i}.
\]
Consequently, step $k$ of the branching-selection process preserves order $\preccurlyeq$ if $\phi_k \leq \psi_k$.
\end{proof}

This lemma implies that branching random walks with selection and branching random walk with killing can be coupled in an increasing fashion for the order $\preccurlyeq$, as soon as there are at any time $k \leq n$ more individuals in one process than in the other. The main result of the section is the following estimate on the extremal positions in the branching random walk with selection.
\begin{theorem}
\label{thm:Thbehaviour}
Assuming \eqref{eqn:reproduction}, \eqref{eqn:critical} and \eqref{eqn:integrability}, for any continuous positive function $h$ we have
\[
  \lim_{n \to +\infty} \frac{M^h_n}{n^{1/3}} = h_0 - \frac{\pi^2 \sigma^2}{2} \int_0^1 \frac{ds}{h_s^2} \quad \mathrm{and} \quad \lim_{n \to +\infty} \frac{m^h_n}{n^{1/3}} = h_0-h_1 - \frac{\pi^2 \sigma^2}{2} \int_0^1 \frac{ds}{h_s^2} \quad \mathrm{a.s.}
\]
\end{theorem}

\begin{remark}
It is worth noting that choosing $h$ as a constant, Theorem \ref{thm:Thbehaviour} provides information on the Brunet-Derrida's $N$-BRW, on the time scale $\frac{(\log N)^3}{h^3}$. Letting $h \to 0$, we study the asymptotic behaviour of the $N$-BRW on a typical time scale.
\end{remark}

The proof of Theorem \ref{thm:Thbehaviour} is based on the construction of an increasing coupling existing between $(\T^h_{(n)},V)$ and approximatively $e^{h_0 n^{1/3}}$ independent branching random walks with a killing boundary $n^{1/3}f_{./n}$. Using Lemma \ref{lem:coupling}, it is enough to bound the size of the population at any time in the branching random walks with a killing boundary to prove the coupling. In a first time, we bound from below the branching random walk with selection by $e^{(h_0-2\epsilon)n^{1/3}}$ independent branching random walks with a killing boundary.
\begin{lemma}
\label{lem:couplingfrombelow}
We assume that \eqref{eqn:reproduction} and \eqref{eqn:critical} hold. For any positive continuous function $h$ and $\epsilon>0$, there exists a coupling between $(\T^h_{(n)},V)$ and i.i.d. branching random walks $((\T^j,V^j), j \geq 1)$ such that almost surely for any $n \geq 1$ large enough, we have
\begin{equation}
  \label{eqn:couplingfrombelow}
  \forall k \leq n, \sum_{\substack{u \in \T_{(n)}^h\\ |u| = k}} \delta_{V(u)} \succcurlyeq \sum_{j=1}^{e^{(h_0-2\epsilon)n^{1/3}}} \sum_{\substack{u \in \T^j\\ |u|=k}} \ind{V^j(u_i) \geq (f_{i/n}-\epsilon)n^{1/3}, i \leq k} \delta_{V^j(u)}.
\end{equation}
\end{lemma}

\begin{proof}
Let $n \in \N$ and $\epsilon>0$, we denote by $p = \floor{e^{(h_0-2\epsilon) n^{1/3}}}$ and by $\tilde{\T}^{(n)}_{f-\epsilon}$ the disjoint union of ${\T^j}^{(n)}_{f-\epsilon}$ for $j \leq p$. For $u \in \tilde{\T}^{(n)}_{f-\epsilon}$, we write $V(u) = V^j(u)$ if $u \in \T^j$. By Lemma \ref{lem:coupling}, it is enough to prove that almost surely, for any $n \geq 1$ large enough we have
\[
  \forall k \leq n, \log \# \left\{ u \in \tilde{\T}^{(n)}_{f-\epsilon} : |u|=k \right\} \leq n^{1/3} h_{k/n}.
\]

We first prove that with high probability, no individual in $\tilde{\T}^{(n)}_{f-\epsilon}$ crosses the frontier $(g_{k/n}-\epsilon)n^{1/3}$ at some time $k \leq n$. By Lemma \ref{lem:frontier}, we have
\begin{multline*}
  \limsup_{n \to +\infty} n^{-1/3} \log \P\left( \exists u \in \tilde{\T}^{(n)}_{f-\epsilon} : V(u) \geq (g_{|u|/n} - \epsilon) n^{1/3} \right)\\
  \leq \limsup_{n \to +\infty} n^{-1/3} \log \left(p \P\left( \exists u \in \T^{(n)}_{f-\epsilon} : V(u) \geq (g_{|u|/n} - \epsilon)n^{1/3}\right) \right) \\
  \leq h_0-2\epsilon - \inf_{t \in [0,1]} g_t-\epsilon + \frac{\pi^2 \sigma^2}{2} \int_0^t \frac{ds}{(g_s - f_s)^2} = - \epsilon.
\end{multline*}
Using the Borel-Cantelli lemma, almost surely for any $n \geq 1$ large enough and $u \in \tilde{\T}^{(n)}_{f-\epsilon}$, we have $V(u) \leq (g_{|u|/n} - \epsilon)n^{1/3}$.

By this result, almost surely, for $n \geq 1$ large enough and for $k \leq n$, the size of the $k^\mathrm{th}$ generation in $\tilde{\T}^{(n)}_{f-\epsilon}$ is given by
\[
  Z^{(n)}_k = \sum_{u \in \tilde{T}^{(n)}_{f-\epsilon}} \ind{|u|=k} \ind{V(u_j) \leq (g_{j/n}-\epsilon)n^{1/3}, j \leq k}.
\]
Using the Markov inequality, we have
\[
  \P\left( \exists k \leq n: Z^{(n)}_k \geq e^{n^{1/3}h_{k/n}} \right) \leq \sum_{k=1}^n e^{-n^{1/3} h_{k/n}} \E\left[ Z^{(n)}_k \right].
\]
We now provide an uniform upper bound for $\E(Z^{(n)}_k)$. Applying Lemma \ref{lem:manytoone}, for any $1 \leq k \leq n$ we have
\begin{align*}
  \E\left[ Z^{(n)}_k \right]
  &\leq p \E\left[ e^{-S_k} \ind{S_j \in \left[ (f_{j/n}-\epsilon) n^{1/3}, (g_{j/n}-\epsilon)n^{1/3} \right]} \right]\\
  &\leq p e^{-(f_{k/n}-\epsilon)n^{1/3}} \P\left( S_j \in \left[ (f_{j/n}-\epsilon) n^{1/3}, (g_{j/n}-\epsilon)n^{1/3} \right] , j \leq k\right).
\end{align*}
Let $A \in \N$. For any $a \leq A$ we write $m_a = \floor{na/A}$ and $\underline{f}_{a,A} = \inf_{s \in [a/A,(a+1)/A]} f_s$. For any $k \in (m_a,m_{a+1}]$, applying the Markov property at time $m_a$ and Theorem \ref{thm:mogulskii} we have
\[
  \E\left[ Z^{(n)}_k \right] \leq \exp\left[(h_0-2\epsilon)n^{1/3}-n^{1/3}\left(\underline{f}_{a,A} - \epsilon + \frac{\pi^2 \sigma^2}{2} \int_0^{a/A} \frac{ds}{h_s^2} \right)\right] 
\]
As $h_0 = f_t + h_t + \frac{\pi^2 \sigma^2}{2} \int_0^t \frac{ds}{h_s^2}$, letting $A \to +\infty$ we have
\[
  \limsup_{n \to +\infty} n^{-1/3} \log \P\left( \exists k \leq n : Z^{(n)}_k \geq e^{n^{1/3}h_{k/n}} \right) \leq - \epsilon.
\]
Consequently, applying the Borel-Cantelli lemma again, for any $n \geq 1$ large enough we have
\[
  \forall k \leq n, \log \# \left\{ u \in \tilde{\T}^{(n)}_{f-\epsilon} : |u|=k \right\} \leq n^{1/3} h_{k/n}
\]
which concludes the proof, by Lemma \ref{lem:coupling}.
\end{proof}

Similarly, we prove that the branching random walk with selection is bounded from above by $\floor{e^{(h_0+2\epsilon) n^{1/3}}}$ independent branching random walks with a killing boundary.
\begin{lemma}
\label{lem:couplingfromabove}
We assume \eqref{eqn:reproduction}, \eqref{eqn:critical} and \eqref{eqn:integrability} hold. For any continuous positive function $h$ and $\epsilon>0$, there exists a coupling between $(\T^h_{(n)},V)$ and i.i.d. branching random walks $((\T^j,V^j), j \geq 1)$ such that almost surely for any $n \geq 1$ large enough we have
\begin{equation}
  \label{eqn:couplingfromabove}
  \forall k \leq n, \sum_{\substack{u \in \T_{(n)}^h\\ |u| = k}} \delta_{V(u)} \preccurlyeq \sum_{j=1}^{e^{(h_0+2\epsilon) n^{1/3}}} \sum_{\substack{u \in \T^j\\ |u|=k}} \ind{V^j(u_i) \geq (f_{i/n}-\epsilon)n^{1/3}, i \leq k} \delta_{V^j(u)}.
\end{equation}
\end{lemma}

\begin{proof}
Let $n \in \N$ and $\epsilon>0$, we denote by $p = \floor{e^{(h_0+2\epsilon) n^{1/3}}}$ and by $\tilde{\T}^{(n)}_{f-\epsilon}$ the disjoint union of ${\T^j}^{(n)}_{f-\epsilon}$ for $j \leq p$. For $u \in \tilde{\T}^{(n)}_{f-\epsilon}$, we write $V(u) = V^j(u)$ if $u \in \T^j$. Similarly to the previous lemma, the key tool is a bound from below of the size of the population at any time in $\tilde{\T}^{(n)}_{f-\epsilon}$. For any $1 \leq k \leq n$, we set
\begin{multline*}
  Z^{(n)}_k = \sum_{u \in \tilde{\T}^{(n)}_{f-\epsilon}} \ind{|u|=k} \ind{V(u_j) \leq (g_{j/n} - \epsilon)n^{1/3}, j \leq k} \quad \mathrm{and}\\
  \tilde{Z}^{(n)}_k = \sum_{u \in \tilde{\T}^{(n)}_{f-\epsilon}} \ind{|u|=k} \ind{V(u) \geq f_1 n^{1/3}} \ind{V(u_j) \leq (g_{j/n} - \epsilon)n^{1/3}, j \leq k}.
\end{multline*}

For any $t \in (0,1)$, applying Corollary \ref{cor:sizepop}, we have
\[
  \limsup_{n \to +\infty} n^{-1/3} \log \P\left[ \tilde{Z}^{(n)}_\floor{nt} \leq e^{(h_t+\epsilon) n^{1/3}} \right] \leq -3\epsilon.
\]
Let $A \in \N$, for $a \leq A$ we set $m_a = \floor{na/A}$. By the Borel-Cantelli lemma, almost surely, for any $n \geq 1$ large enough we have
\[
  \forall a \leq A, \log \tilde{Z}^{(n)}_{m_a} \geq n^{1/3} (h_\frac{a}{A} + \epsilon).
\]

We extend this result into an uniform one. To do so, we notice that Theorem \ref{thm:fzfhs} implies there exists $r>0$ small enough and $\lambda > 0$ large enough such that
\[
  \inf_{n \in \N} \P\left[ \exists |u| = n : \forall k \leq n, V(u_k) \geq - \lambda n^{1/3} \right] > r.
\]
Consequently, every individual alive at time $m_a$ above $f_{a/A}n^{1/3}$ start an independent branching random walk, which has probability at least $r$ to have a descendant at time $m_{a+1}$ which stayed at any time in $k \in [m_a,m_{a+1}]$ above $(f_{a/A}-\lambda A^{-1/3} n^{1/3}$. Choosing $A>0$ large enough, conditionally on $\calF_{m_a}$, $\inf_{k \in [m_a,m_{a+1}]} Z^{(n)}_k$ is stochastically bounded from below by a binomial variable with parameters $Z_{m_a}^{(n)}$ and $r$. We conclude from an easy large deviation estimate and the Borel-Cantelli lemma again, that almost surely for $n \geq 1$ large enough we have
\[
  \forall k \leq n, \log Z^{(n)}_k \geq n^{1/3} h_{k/n}.
\]
Applying Lemma \ref{lem:coupling}, we conclude that
\[
  \forall k \leq n, \sum_{\substack{u \in \T_{(n)}^h \\ |u| = k}} \delta_{V(u)} \preccurlyeq \sum_{\substack{u \in \tilde{\T}^{(n)}_{f-\epsilon}\\ |u|=k}} \delta_{V(u)}.
\]
\end{proof}

Using Lemmas \ref{lem:couplingfrombelow} and \ref{lem:couplingfromabove}, we easily bound the maximal and the minimal displacement in the branching random walk with selection.
\begin{proof}[Proof of Theorem \ref{thm:Thbehaviour}]
The proof is based on the observation that for any $x_1 \geq x_2 \geq \cdots \geq x_p$ and $y_1 \geq y_2 \geq \cdots \geq y_q$, if $\sum_{j=1}^p \delta_{x_j} \preccurlyeq \sum_{j=1}^q \delta_{y_j}$ then $p \leq q$, $x_1 \leq y_1$ and $x_p \leq y_p$.

Let $n \in \N$ and $\epsilon>0$, we denote by $\check{p} = \floor{e^{(h_0-2\epsilon) n^{1/3}}}$ and by $\hat{p} = \floor{e^{(h_0+2\epsilon) n^{1/3}}}$. Given $((\T^j,V^j), j \in \N)$ i.i.d. branching random walks, we set $\check{\T}^{(n)}_{f-\epsilon}$ (respectively $\hat{\T}^{(n)}_{f-\epsilon}$) the disjoint union of ${\T^j}^{(n)}_{f-\epsilon}$ for $j \leq \check{p}$ (resp. $j \leq \hat{p}$). For $u \in \hat{\T}^{(n)}_{f-\epsilon}$, we write $V(u) = V^j(u)$ if $u \in \T^j$. By Lemmas \ref{lem:couplingfrombelow} and \ref{lem:couplingfromabove}, we have
\[
  \max_{u \in \check{\T}^{(n)}_{f-\epsilon}, |u|=n} V(u) \leq M^h_n \leq \max_{u \in \hat{\T}^{(n)}_{f-\epsilon}, |u|=n} V(u).
\]

For any $\delta > -h_0$, we denote by $g^\delta$ the solution of the differential equation
\[
  g^\delta_t + \frac{\pi^2 \sigma^2}{2} \int_0^t \frac{ds}{(g^\delta_s - f_s)^2} = h_0+\delta.
\]
We observe that $g^\delta$ is well-defined on $[0,1]$ for $\delta$ in a neighbourhood of $0$. We notice that $g^0 = g$ and that $\delta \mapsto g^\delta$ is continuous with respect to the uniform norm. Moreover
\begin{align*}
  \P\left( \max_{u \in \hat{\T}^{(n)}_{f-\epsilon}, |u|=n} V(u) \geq g^\delta_1 n^{1/3} \right)
  &\leq \P\left( \exists u \in \hat{\T}^{(n)}_{f-\epsilon} : V(u) \geq g^\delta_{|u|/n} n^{1/3} \right)\\
  &\leq \hat{p} \P\left( \exists |u| \leq n : V(u) \geq g^\delta_{|u|/n} n^{1/3} \right).
\end{align*}
Consequently, using Lemma \ref{lem:frontier}, we have
\[
  \limsup_{n \to +\infty} n^{-1/3} \log \P\left( \max_{u \in \hat{\T}^{(n)}_{f-\epsilon}, |u|=n} V(u) \geq g^\delta_1 n^{1/3} \right)
  \leq h_0 + 2 \epsilon - \inf_{t \in [0,1]} g^\delta_t + \frac{\pi^2 \sigma^2}{2} \int_0^t \frac{ds}{(g^\delta_s-f_s + \epsilon)^2}.
\]
For any $\delta > 0$, for any $\epsilon>0$ small enough we have
\[
  \limsup_{n \to +\infty} n^{-1/3} \log \P\left( M^h_n \geq g^\delta_1 n^{1/3} \right) < 0.
\]
By the Borel-Cantelli lemma, we have $\limsup_{n \to +\infty} \frac{M^h_n}{n^{1/3}} \leq g^\delta_1$ a.s. Letting $\delta \to 0$ concludes the proof of the upper bound of the maximal displacement.

To obtain a lower bound, we notice that
\begin{align*}
  \P\left( M^h_n \leq (g^\delta_1-2\epsilon) n^{1/3} \right)&\leq \P\left( \max_{u \in \check{\T}^{(n)}_{f-\epsilon}, |u|=n} V(u) \leq (g^\delta_1-2\epsilon) n^{1/3} \right)\\
  &\leq \P\left( \max_{|u|=n} V(u) \leq (g^\delta_1-2\epsilon) n^{1/3} \right)^{\check{p}}.
\end{align*}
We only consider individuals that stayed at any time $k \leq n$ between the curves $n^{1/3}(f_{k/n}-\epsilon)$ and $n^{1/3}(g^{-\delta}_{k/n}-\epsilon)$, applying Corollary \ref{cor:lowboundproba}, for any $\delta > 0$ small enough, for any $\epsilon>0$ small enough, we have
\begin{multline*}
  \liminf_{n \to +\infty} n^{-1/3} \log \P\left( \exists |u| = n : V(u) \geq (g^{-\delta}_1-2\epsilon) n^{1/3} \right)\\
  \geq - \sup_{t \in [0,1]} g^{-\delta}_t - \epsilon + \frac{\pi^2 \sigma^2}{2} \int_0^t \frac{ds}{(g^\delta_s - f_s)^2} \geq \epsilon - h_0 + \delta.
\end{multline*}
As a consequence,
\[
  \liminf_{n \to +\infty} n^{-1/3} \log \left(- \log \P\left( M^h_n \leq (g^\delta_1-2\epsilon) n^{1/3} \right) \right) \geq \delta - \epsilon.
\]
For any $\delta > 0$ small enough, for any $\epsilon>0$ small enough, applying the Borel-Cantelli lemma we have
\[
  \liminf_{n \to +\infty} \frac{M^h_n}{n^{1/3}} \geq g^\delta_1 - 2 \epsilon \quad \mathrm{a.s.}
\]
Letting $\epsilon \to 0$ then $\delta \to 0$ concludes the almost sure asymptotic behaviour $M^h_n$.

We now bound $m^h_n$. By Lemma \ref{lem:couplingfromabove}, almost surely for $n \geq 1$ large enough, the $\floor{e^{n^{1/3} h_1}}^\mathrm{th}$ rightmost individual at generation $n$ in $\hat{\T}^{(n)}_{f-\epsilon}$ is above $m^h_n$. Therefore for any $x \in \R$, almost surely for $n \geq 1$ large enough,
\[
  \ind{m^h_n \geq x n^{1/3}} \leq \ind{\# \left\{ u \in \hat{\T}^{(n)}_{f-\epsilon} : |u| =n, V(u) \geq xn^{1/3} \right\} \geq e^{h_1 n^{1/3}}}. 
\]
Let $\delta > 0$. By Lemma \ref{lem:frontier}, we have
\[
  \limsup_{n \to +\infty} n^{-1/3} \log \P\left( \exists u \in \hat{\T}^{(n)}_{f-\epsilon} : V(u) \geq (g^\delta_{k/n}-\epsilon)n^{1/3} \right) \leq h_0 - (h_0 + \delta - \epsilon).
\]
Consequently, for any $\delta > 0$, for any $\epsilon>0$ small enough, almost surely for $n \geq 1$ large enough the population in $\hat{\T}^{(n)}_{f-\epsilon}$ at time $k$ belongs to $I^{(n)}_k$. We write
\[
  Z^{(n)}(x) = \sum_{u \in \hat{\T}^{(n)}_{f-\epsilon}} \ind{|u|=n} \ind{V(u) \geq xn^{1/3}} \ind{V(u_j) \leq (g^\delta_{j/n}-\epsilon)n^{1/3}, j \leq n}.
\]
By Lemma \ref{lem:meannumber}, we have
\begin{align*}
  \limsup_{n \to +\infty} n^{-1/3} \log \E\left[Z^{(n)}(x)\right] &\leq h_0 - \left( x + \frac{\pi^2 \sigma^2}{2} \int_0^t \frac{ds}{(g^\delta_s - f_s)^2} \right)\\
  &\leq g^\delta_1 - \delta -x.
\end{align*}
Using the Markov inequality, for any $\delta > 0$, for any $n \geq 1$ large enough we have $Z^{(n)}(g^\delta_1 - h_1) \leq e^{h_1 n^{1/3}}$, which leads to
\[
  \limsup_{n \to +\infty} \frac{m^h_n}{n^{1/3}} \leq g^\delta_1 - h_1 \quad \mathrm{a.s.}
\]
Letting $\delta \to 0$ concludes the proof of the upper bound of $m^h_n$.

The lower bound is obtained in a similar fashion. For any $\zeta > 0$, we write $k = \floor{\zeta n^{1/3}}$. Almost surely, for $n \geq 1$ large enough we have
\[
  \sum_{\substack{u \in \check{\T}^{(n)}_{f-\epsilon}\\ |u| = n-k}} \delta_{V(u)} \preccurlyeq \sum_{\substack{u \in \T^h_{(n)} \\ |u| = n-k}} \delta_{V(u)}.
\]
This inequality is not enough to obtain a lower bound on $m^h_n$, as there are less than $e^{h_1 n^{1/3}}$ individuals alive in $\check{\T}^{(n)}_{f-\epsilon}$ at generation $n-k$. Therefore, starting from generation $n-k$, we start a modified branching-selection procedure that preserve the order $\preccurlyeq$ and guarantees there are $\floor{e^{h_1 n^{1/3}}}$ individuals alive at generation $n$.

In a first time, we bound from below the size of the population alive at generation $n-k$. We write, for $\delta > 0$ and $\eta > 0$
\[
  X^{(n)} = \sum_{u \in \check{\T}^{(n)}_{f-\epsilon}} \ind{|u| = n-k} \ind{V(u_j) \leq (g^{-\delta}_{j/n}-\epsilon)n^{1/3}, \xi(u_j) \leq e^{\eta n^{1/3}}, j \leq n-k}.
\]
By Lemma \ref{lem:secondorder}, we have
\[
  \liminf_{n \to +\infty} n^{-1/3} \log \E(X^{(n)}) \geq h_0 - 2\epsilon -\left( (f_1-\epsilon) + \frac{\pi^2 \sigma^2}{2} \int_0^1 \frac{ds}{(g^{-\delta}_s - f_s)^2}\right) = \delta - \epsilon + (g^{-\delta}_1 - f_1).
\]
Consequently, using the fact that for $\check{p}$ i.i.d. random variables $(X_j)$, we have
\[
  \P\left( \sum_{j=1}^{\check{p}} X_j \leq \check{p} \E(X_1)/2 \right) \leq \frac{4\E(X^2_1)}{\check{p} \E(X_1)^2},
\]
for any $\epsilon>0$ and $\delta > 0$ small enough enough, Lemma \ref{lem:secondorder} leads to
\[
  \limsup_{n \to +\infty} n^{-1/3} \log \P\left( X^{(n)} \leq e^{((g^{-\delta}_1-f_1) + \delta)n^{1/3}} \right) \leq \eta + h_0 - \delta - \epsilon-(h_0-2\epsilon).
\]
For any $\xi>0$, choosing $\delta > 0$ small enough, and $\epsilon>0$ and $\eta > 0$ small enough, we conclude by the Borel-Cantelli lemma that almost surely, for $n \geq 1$ large enough
\[
  \#\left\{ u \in \check{\T}^{(n)}_{f-\epsilon} : |u| = n- k \right\} \geq \exp\left(n^{1/3} (h_1 - \xi)\right).
\]

In a second time, we observe by \eqref{eqn:reproduction} there exists $a>0$ and $\rho>1$ such that 
\[\E\left( \sum_{|u|=1} \ind{V(u) \geq -a} \right) > \rho.\]
We consider the branching-selection process that starts at time $n-k$ with the population of the $(n-k)^\mathrm{th}$ generation of $\check{\T}^{(n)}$, in which individuals reproduce independently according to the law $\calL$, with the following selection process: an individual is erased if it belongs to generation $n-k+j$ and is below $n^{1/3}f_{(n-k)/n} - j a$, or if it is not one of the $e^{n^{1/3}h_{(n-k+j)/n}}$ rightmost individuals. By Lemma \ref{lem:coupling}, this branching-selection process stays at any time $n-k \leq j \leq n$ below $(\T^h_{(n)},V)$ for the order $\preccurlyeq$. Moreover, by definition, the leftmost individual alive at time $n$ is above $n^{1/3}(f_{(n-k)/n}-\epsilon- a \zeta)$.

We now bound the size of the population in this process. We write $(X_j, j \in \N)$ for a sequence of i.i.d. random variables with the same law as $ \sum_{|u|=1} \ind{V(u) \geq -a}$. By Cramér's theorem, there exists $\lambda > 0$ such that for any $n \in \N$, we have
\[
  \P\left( \sum_{k=1}^n X_j \leq n\rho \right) \leq e^{-\lambda n}.
\]
Consequently, the probability that there exists $j \in [n-k,n]$ such that the size of the population at time $j$ in the branching-selection process is less than $\min\left(\rho^{k+j-n} e^{(h_{(n-k)/n}-\xi)n^{1/3}}, e^{h_{j/n} n^{1/3}}\right)$ decays exponentially fast with $n$. Applying the Borel-Cantelli lemma, for any $\zeta>0$, there exists $\xi>0$ such that almost surely for $n \geq 1$ large enough, the number of individuals alive at generation $n$ in the bounding branching-selection process is $\floor{e^{h_1 n^{1/3}}}$. On this event, $m^h_n$ is greater than the minimal position in this process. We conclude, letting $n$ grows to $+\infty$ then $\epsilon$ and $\zeta$ decrease to $0$ that
\[
  \liminf_{n \to +\infty} \frac{m^h_n}{n^{1/3}} \geq h_0 - h_1 - \frac{\pi^2 \sigma^2}{2}\int_0^1 \frac{ds}{h_s^2} \quad \mathrm{a.s.}
\]
completing the proof of Theorem \ref{thm:Thbehaviour}.
\end{proof}

An application of Theorem \ref{thm:Thbehaviour} leads to Theorem \ref{thm:main}.
\begin{proof}[Proof of Theorem \ref{thm:main}]
Let $a > 0$, we denote by $\phi : n \mapsto \floor{e^{a n^{1/3}}}$ and by $(\T^\phi,V)$ the branching random walk with selection of the $\phi(n)$ rightmost individuals at generation $n$. For $n \in \N$ we write
\[
  M^\phi_n = \max_{ u \in \T^\phi, |u| = n} V(u) \quad \mathrm{and} \quad m^\phi_n = \min_{u \in \T^\phi, |u| = n} V(u).
\]

Let $\epsilon>0$ and $n \in \N$, we set $k = \floor{n\epsilon}$ and $h : t \mapsto a(t+\epsilon)^{1/3}$. We note that by Lemma \ref{lem:coupling}, for any two continuous non-negative functions $h_1 \leq h_2$, and $k \leq n$ we have
\[
  \sum_{\substack{u \in \T^{h_1}_{(n)} \\ |u| = k}} \delta_{V(u)} \preccurlyeq \sum_{\substack{u \in \T^{h_2}_{(n)}\\ |u| =k}} \delta_{V(u)}.
\]
As a consequence, for any $n \in \N$ and $\epsilon>0$, we couple the branching random walk with selection $(\T^\phi,V)$ with two branching random walks with selection $(\T_{(n)}^{h,+},V)$ and $(\T_{(n)}^{h,-},V)$ in a way that
\begin{equation}
  \label{eqn:couplingidiot}
  \sum_{\substack{u \in \T_{(n)}^{h,-}\\ |u| = n-k}} \delta_{V(u) + m^\phi_k} \preccurlyeq \sum_{\substack{u \in \T^\phi\\ |u| = n}} \delta_{V(u)} \preccurlyeq \sum_{\substack{u \in \T_{(n)}^{h,-}\\ |u| = n-k}} \delta_{V(u) + M^\phi_k},
\end{equation}
using the fact that the population at time $k$ in $\T^\phi$ is between $m^\phi_k$ and $M^\phi_k$.

Applying Theorem \ref{thm:Thbehaviour}, we have
\[
  \limsup_{n \to +\infty} \frac{M^\phi_n - M^\phi_k}{n^{1/3}} \leq \limsup_{n \to +\infty} \frac{M^h_{n-k}}{n^{1/3}} \leq a\epsilon^{1/3} - \frac{\pi^2 \sigma^2}{2} \int_0^{1-\epsilon} \frac{ds}{(a(s+\epsilon)^{1/3})^2} \quad \mathrm{a.s.}
\]
as well as
\[
  \liminf_{n \to +\infty} \frac{m^\phi_n - m^\phi_k}{n^{1/3}} \geq \liminf_{n\to +\infty} \frac{m^h_{n-k}}{n^{1/3}} \geq - a - \frac{\pi^2 \sigma^2}{2} \int_0^{1-\epsilon} \frac{ds}{(a(s+\epsilon)^{1/3})^2} \quad \mathrm{a.s.}
\]
As $\lim_{\epsilon \to 0} \int_0^{1-\epsilon} \frac{ds}{(a(s+\epsilon)^{1/3})^2} = \frac{3}{a^2}$, for any $\delta > 0$, for any $\epsilon>0$ small enough we have
\[
  \limsup_{n \to +\infty} \frac{M^\phi_n - M^\phi_\floor{\epsilon n}}{n^{1/3}} \leq - \frac{3\pi^2 \sigma^2}{2a^2} + \delta  \quad \mathrm{a.s.}
\]
We set $p = \floor{- \frac{\log n}{\log \epsilon}}$, and observe that
\begin{align*}
  \frac{M^\phi_n}{n^{1/3}} &= \frac{1}{n^{1/3}} \sum_{j=0}^{p-2} \left(M^\phi_\floor{\epsilon^j n} - M^\phi_{\floor{\epsilon^{j+1} n}} \right) + \frac{M^\phi_\floor{\epsilon^{p-1} n}}{n^{1/3}}\\
  &\leq \sum_{j=0}^{p-2} \epsilon^{j/3} \frac{M^\phi_\floor{\epsilon^j n} - M^\phi_\floor{\epsilon^{j+1} n}}{(\epsilon^j n)^{1/3}} + \frac{\sup_{j \leq \epsilon^{-2}} M^\phi_j}{n^{1/3}}.
\end{align*}
Using a straightforward adaptation of the Cesàro lemma, we obtain
\[
  \limsup_{n \to +\infty} \frac{M^\phi_n}{n^{1/3}} \leq \frac{- \frac{3\pi^2 \sigma^2}{2a^2} + \delta}{1 - \epsilon^{1/3}}  \quad \mathrm{a.s.}
\]
Letting $\epsilon \to 0$ then $\delta \to 0$ we have
\begin{equation}
  \label{eqn:ubub}
  \limsup_{n \to +\infty} \frac{M^\phi_n}{n^{1/3}} \leq - \frac{3\pi^2 \sigma^2}{2 a^2} \quad \mathrm{a.s.}
\end{equation}

Similarly, for any $\delta > 0$, for any $\epsilon>0$ small enough we have
\[
  \liminf_{n \to +\infty} \frac{m^\phi_n - m^\phi_\floor{\epsilon n}}{n^{1/3}} \geq -a - \frac{3\pi^2 \sigma^2}{2a^2} - \delta  \quad \mathrm{a.s.}
\]
Setting $p = \floor{- \frac{\log n}{\log \epsilon}}$ and observing that
\[
  \frac{m^\phi_n}{n^{1/3}} \geq \sum_{j=0}^{p-2} \epsilon^{j/3} \frac{m^\phi_\floor{\epsilon^j n} - m^\phi_\floor{\epsilon^{j+1} n}}{(\epsilon^j n)^{1/3}} + \frac{\inf_{j \leq \epsilon^{-2}} m^\phi_j}{n^{1/3}},
\]
we use again the Cesàro lemma to obtain, letting $\epsilon$ then $\delta$ decrease to 0,
\begin{equation}
  \label{eqn:lblb}
  \liminf_{n \to +\infty} \frac{m^\phi_n}{n^{1/3}} \geq - a - \frac{3\pi^2 \sigma^2}{2 a^2} \quad \mathrm{a.s.}
\end{equation}

To obtain the other bounds, we observe that \eqref{eqn:couplingidiot} also leads to
\[
  \liminf_{n \to +\infty} \frac{M^\phi_n}{n^{1/3}} \geq \liminf_{n \to +\infty} \frac{M^h_{n-k}+ m^\phi_k}{n^{1/3}} \geq - \frac{\pi^2 \sigma^2}{2a^2} \int_0^{1-\epsilon} \frac{ds}{(s + \epsilon)^{2/3}} -\left(a + \frac{3\pi^2 \sigma^2}{2 a^2} \right) \epsilon^{1/3} \quad \mathrm{a.s.}
\]
by Theorem \ref{thm:Thbehaviour} and \eqref{eqn:lblb}. Letting $\epsilon \to 0$ we have
\[
  \liminf_{n \to +\infty} \frac{M^\phi_n}{n^{1/3}} \geq - \frac{3\pi^2 \sigma^2}{2a^2} \quad \mathrm{a.s.}
\]
Similarly, we have
\[
  \limsup_{n \to +\infty} \frac{m^\phi_n}{n^{1/3}} \leq \limsup_{n \to +\infty} \frac{m^h_{n-k}+ M^\phi_k}{n^{1/3}} \leq -a - \frac{\pi^2 \sigma^2}{2a^2} \int_0^{1-\epsilon} \frac{ds}{(s + \epsilon)^{2/3}} \quad \mathrm{a.s.}
\]
using Theorem \ref{thm:Tfbehaviour} and \eqref{eqn:ubub}. We let $\epsilon \to 0$ to obtain
\[
  \limsup_{n \to +\infty} \frac{m^\phi_n}{n^{1/3}} \leq - a - \frac{3\pi^2 \sigma^2}{2 a^2} \quad \mathrm{a.s.}
\]
\end{proof}

The careful reader will notice that, for almost any $a \in \R$ there exist $\bar{a} \neq a$ such that
\[
  a + \frac{3\pi^2 \sigma^2}{2a^2} = \bar{a} + \frac{3 \pi^2 \sigma^2}{2 \bar{a}^2}.
\]
With these notation, both the branching random walk with selection of the $e^{a n^{1/3}}$ rightmost individuals at generation $n$ and the branching random walk with selection of the $e^{\bar{a}n^{1/3}}$ rightmost ones are coupled, between times $\epsilon n$ and $n$ with branching random walks with the same killing barrier
\[
  f : t \in [\epsilon, 1] \mapsto  \left(a + \frac{3\pi^2 \sigma^2}{2a^2}\right) t^{1/3},
\]
the difference between the processes being the number of individuals initially alive in the processes, respectively $e^{a (\epsilon n)^{1/3}}$ and $e^{\bar{a} (\epsilon n)^{1/3}}$.

\paragraph*{Acknowledgements.} I would like to thank Zhan Shi for having started me on this topic and for his constant help, as well as Shi Wanlin for pointing out typos in the previous versions.

\nocite{*}
\bibliographystyle{plain}\def\cprime{$'$}

\end{document}